\documentclass[11pt,leqno]{article}
\usepackage{amsmath, amscd, amsthm, amssymb, graphics, xypic, mathrsfs, setspace, fancyhdr, bm, pdfsync, enumitem, mathptmx}
\usepackage[usenames, dvipsnames, svgnames, table]{xcolor}
\usepackage[letterpaper,top=1.05in, bottom=1.05in, left=1.05in, right=1.05in]{geometry}
\usepackage[colorlinks=true,pagebackref=true]{hyperref}
\hypersetup{backref}

\newcommand{\Spec}{\operatorname{Spec}}

\newcommand{\sma}{{\scriptstyle{\wedge}\,}}

\newcommand{\GW}{\mathbf{GW}}

\newcommand{\real}{{\mathbb R}}
\newcommand{\PP}{\mathbb{P}}

\newcommand{\cplx}{{\mathbb C}}

\newcommand{\Z}{{\mathbb Z}}
\newcommand{\N}{{\mathbb N}}
\newcommand{\A}{{\mathbb A}}

\newcommand{\OO}{{\mathcal O}}
\newcommand{\Osc}{{\mathscr{O}}}
\newcommand{\Esc}{{\mathscr{E}}}

\newcommand{\aone}{{\mathbb A}^1}
\newcommand{\pone}{{\mathbb P}^1}

\newcommand{\gm}[1]{{{\mathbb G}_{m}^{#1}}}

\renewcommand{\Lc}{{\mathcal L}}

\newcommand{\et}{\text{\'et}}

\newcommand{\bpi}{\bm{\pi}}
\newcommand{\piaone}{{\bpi}^{\aone}}

\newcommand{\Nis}{{\operatorname{Nis}}}
\newcommand{\Zar}{\operatorname{Zar}} 

\newcommand{\CHW}{{\widetilde{\mathrm{CH}}}}
\newcommand{\CH}{{\mathrm{CH}}}
\newcommand{\Ch}{{\mathrm{Ch}}}

\newcommand{\Sm}{\mathrm{Sm}}

\newcommand{\Ab}{\mathrm{Ab}}
\newcommand{\Gr}{\mathrm{Gr}}

\newcommand{\K}{{{\mathbf K}}}

\newcommand{\KMW}{\K^{\mathrm{MW}}}
\newcommand{\KM}{\K^{\mathrm M}}
\newcommand{\Pic}{\operatorname{Pic}}

\newcommand{\Hr}{{\mathrm {H}}}

\newcommand{\Ibf}{\mathbf{I}}

\newcommand{\ZZ}{\Z}

\newcommand{\Rb}{\mathbb{R}}
\newcommand{\Fsc}{\mathscr{F}}
\newcommand{\Wr}{\mathrm{W}}
\newcommand{\Kr}{\mathrm{K}}
\newcommand{\GWr}{\mathrm{GW}}
\newcommand{\GWc}{\mathcal{GW}}
\newcommand{\Abf}{\mathbf{A}}
\newcommand{\Cr}{\mathrm{C}}
\newcommand{\RS}{\mathrm{RS}}
\newcommand{\Ar}{\mathrm{A}}
\newcommand{\Zb}{\mathbb{Z}}
\newcommand{\Cb}{\mathbb{C}}

\newcommand{\Er}{\mathrm{E}}
\newcommand{\Mbf}{\mathbf{M}}

\newcommand{\Fr}{\mathrm{F}}
\newcommand{\Abb}{\mathbb{A}}

\renewcommand{\setminus}{\smallsetminus}

\newcommand{\Addresses}{{
\bigskip
\footnotesize

\noindent A.~Asok, Department of Mathematics, University of Southern California, 3620 S.~Vermont Ave., Los Angeles, CA 90089-2532, United States; E-mail address: asok@usc.edu
\medskip

\noindent J.~Fasel,  Institut Fourier, Université Grenoble Alpes, CS40700, Grenoble 38058 Cedex 9, France; E-mail address: jean.fasel@univ-grenoble-alpes.fr
\medskip 

\noindent S.~Lerbet,  DMA, École normale supérieure, Université PSL, CNRS, 75005 Paris, France; E-mail address: samuel.lerbet@ens.fr

}}

\newcounter{intro}
\theoremstyle{plain}
\newtheorem{thm}{Theorem}[subsection]

\newtheorem{lem}[thm]{Lemma}
\newtheorem{cor}[thm]{Corollary}
\newtheorem{prop}[thm]{Proposition}
\newtheorem*{claim*}{Claim} 

\newtheorem*{question*}{Main question}

\newtheorem*{thm*}{Theorem}
\newtheorem*{problem*}{Problem}

\newtheorem{thmintro}{Theorem}

\theoremstyle{definition}
\newtheorem{defn}[thm]{Definition}

\theoremstyle{remark}
\newtheorem{rem}[thm]{Remark}

\newtheorem{ex}[thm]{Example}

\numberwithin{equation}{subsection}

\usepackage{tikz}
\usetikzlibrary{arrows,matrix,cd}

\begin{document}
\pagestyle{fancy}
\renewcommand{\sectionmark}[1]{\markright{\thesection\ #1}}
\fancyhead{}
\fancyhead[LO,R]{\bfseries\footnotesize\thepage}
\fancyhead[LE]{\bfseries\footnotesize\rightmark}
\fancyhead[RO]{\bfseries\footnotesize\rightmark}
\chead[]{}
\cfoot[]{}
\setlength{\headheight}{1cm}

\author{Aravind Asok\thanks{Aravind Asok was partially supported by National Science Foundation Awards DMS-1802060 and DMS-2101898} \and 
Jean Fasel\thanks{Jean Fasel was partially supported by the ANR project HQDIAG, grant number ANR-21-CE40-0015} \and Samuel Lerbet\thanks{Samuel Lerbet was partially supported by the ANR project HQDIAG, grant number ANR-21-CE40-0015, and by the ANR project CYCLADES, grant number ANR-23-CE40-0011}}

\title{{\bf Splitting vector bundles over real algebraic varieties}}
\date{}


\maketitle

\begin{abstract}
Suppose $X$ is a smooth affine real variety and $\mathscr{E}$ is a vector bundle over $X$.  We analyze the problem of splitting off a free rank one summand from $\mathscr{E}$ in corank $0$ and $1$.  The problem in corank $0$ can be viewed as the search for a real analog of Murthy's celebrated splitting theorem in the algebraically closed case: to wit, beyond the vanishing of the top Chern class in Chow theory, are the obstructions to splitting ``purely topological''?  In a sense, the answer in this case is yes, and we give a proof, using motivic techniques, of a mild extension of the results of Bhatwadekar-Sridharan and Bhatwadekar--Das--Mandal: in corank $0$, vanishing of the top Chern class together with topological splitting is sufficient to guarantee algebraic splitting.

Murthy's splitting conjecture (now a theorem in characteristic $0$) predicts that the vanishing of the top Chern class in Chow theory is the only obstruction to splitting off a free rank $1$ summand from a corank $1$ module over a smooth affine $k$-variety over an algebraically closed field.  We search for a suitable ``real'' analog of this assertion.  We observe that several natural guesses for a ``real'' analog of Murthy's splitting conjecture cannot be true, i.e., that the situation over the real numbers is rather complicated.  
\end{abstract}

\section*{Introduction}
Assume $M$ is a smooth manifold of dimension $d$ and $E$ is a rank $r$ (real) vector bundle on $M$.  If $r > d$, general position arguments imply that $E$ admits a nowhere vanishing section or, equivalently, splits off a free rank $1$ summand.  By contrast, if $r \leq d$ then there is a well-defined {\em primary} obstruction to $E$ splitting off a trivial rank $1$ summand: the vanishing of the Euler class, which can be identified as the cohomology class Poincar\'e dual to a suitably generic section.  If $r = d$, then the vanishing of the primary obstruction is necessary and sufficient to guarantee the existence of a nowhere vanishing section.  If $r < d$, then there are additional obstructions to the existence of a nowhere vanishing section beyond the vanishing of the Euler class.  In the early 1950s, at least under the additional assumption that $E$ is oriented, S.D. Liao identified an explicit secondary obstruction \cite{Liao1954obstructions} for bundles with vanishing Euler class.  The extension of Liao's result to bundles that are not necessarily oriented has attained folklore status.  In the language of modern obstruction theory, however, establishing this generalization presents no essential difficulties.  

For algebraic vector bundles on smooth affine varieties over a field $k$, there is a story that is inspired by the one just described, but where numerous additional ``arithmetic'' complications arise depending on $k$.  From the standpoint of the analogy between topological vector bundles and vector bundles on affine varieties, which goes back to J.-P. Serre, these obstructions are rather efficiently described using the language of motivic homotopy theory. We refer the reader to Section \ref{sec:motivicsplitting} (or \cite{Asok22}) for further information, only mentioning for now that the same terminology applies in that case: the primary obstruction is called the Euler class, and its vanishing guarantees the well-definedness of a (motivic) secondary obstruction.     

Assume that $X$ is a smooth affine variety of dimension $d$ over an algebraically closed field $\overline k$. In that case, a famous theorem of M.P. Murthy (\cite[Remark 3.6 and Theorem 3.7]{Murthy94}) shows that the only obstruction to splitting off a trivial rank $1$ summand from a rank $d$ vector bundle on $X$ is the vanishing of the top Chern class.  

\begin{thm*}[Murthy]
Let $X$ be a smooth affine variety of dimension $d$ over an algebraically closed field $k$, and let $\mathscr{E}$ be a vector bundle of rank $d$ over $X$. Then $\mathscr{E}$ splits off a free summand of rank $1$ if and only if the top Chern class $c_d(\mathscr{E}) \in CH^d(X)$ vanishes.
\end{thm*}

In that case, the top Chern class coincides with the Euler class (e.g., \cite[Remark 33]{Morel08}) and this result can be recast in the above framework. If now $\mathscr{E}$ has corank $1$, i.e., has rank equal to $\mathrm{dim}(X)-1$, the same result holds, at least in characteristic $0$ (\cite{Asok12c},\cite{asokP1stabilizationUnstableMotivicToappear}).

\begin{thm*}
Let $X$ be a smooth affine variety of dimension $d\geq 2$ over an algebraically closed field $k$ of characteristic $0$, and let $\mathscr{E}$ be a vector bundle of rank $d-1$ over $X$. The bundle $\mathscr{E}$ splits off a free summand of rank $1$ if and only if the top Chern class $c_{d-1}(\mathscr{E}) \in CH^{d-1}(X)$ vanishes.
\end{thm*}
     
Swan observed at roughly the inception of the theory of projective modules that the tangent bundle of the real $2$-sphere has trivial Chern class yet fails to be trivial, so ``direct'' analogs of the above theorems are false.  Of course, vanishing of the top Chern class in the Chow theory is a necessary condition for a vector bundle of rank $r$ (with $1 \leq r \leq d$) on a smooth affine variety of dimension $d$ to split off a trivial rank $1$ summand, and viewing the above results as paradigmatic examples of what might be possible, it is natural to ask for ``easily describable'' additional conditions that guarantee splitting results, say in corank $0$ and $1$.  

In corank $0$, which we refer to as the ``critical'' case, the above discussion provided one impetus for the search for an ``Euler class theory'' that accounted for the failure of the top Chern class to detect splitting.  Over the real numbers, the Euler class theory was linked to topological considerations in this case as well.  In \cite{Bhatwadekar99,Bhatwadekar06}, the authors validate the principle that beyond the vanishing of $c_d(E)$, the only obstruction to splitting off a free summand of rank $1$ is the vanishing of the Euler class of the real vector bundle $\mathscr{E}(\real)$ over $X(\real)$ determined by $\mathscr{E}$. In other words, beyond the evidently necessary algebro-geometric obstruction, subsequent conditions depend only on the topology of the real points of $X$. This leads us to study the following question, which can be seen as a real analogue of the splitting theorem in the case of an algebraically closed field.


\begin{question*}
Assume $X$ is a smooth affine $\real$-variety of dimension $d$ and $\Esc$ is a rank $d-1$ vector bundle on $X$. Does $\Esc$ split off a trivial rank $1$ summand if and only if $0 = c_{d-1}(\Esc) \in \CH^{d-1}(X)$ and the real topological vector bundle $\Esc(\real)$ splits off a trivial rank $1$ summand (topologically)?  
\end{question*}

This question unfortunately has a negative answer and we investigate how and why the proposed criterion fails.  Our first result says that the vanishing of the top Chern class of a rank $d-1$ vector bundle $\mathscr{E}$ and of the Euler class of the associated real vector bundle $\mathscr{E}(\real)$ over $X(\real)$ are not even sufficient to guarantee that its (algebraic) Euler class vanishes. 

\begin{thmintro}[See Theorem \ref{thm:realthreefolds}]
There exists an affine open subscheme $U$ of $\mathbb{P}_\real^3$ such that $U(\real)=\mathbb{P}^3(\real)$ and a vector bundle $\mathscr{E}$ of rank $2$ on $U$ having the following properties:
\begin{itemize}[noitemsep,topsep=1pt]
\item $\det(\mathscr{E})$ is trivial.
\item $c_2(\mathscr{E})=0$.
\item $e(\mathscr{E})\neq 0\in{\CHW}^2(U)$.
\item $\mathscr{E}(\real)\simeq \mathbb{P}^3(\real)\times \real^2$.
\end{itemize}
\end{thmintro}

Moreover, even strengthening the hypotheses slightly we cannot guarantee splitting: for instance, given a smooth affine real variety $X$ and an algebraic vector bundle $\mathscr{E}$ on $X$, one might hope that the vanishing of the (algebraic) Euler class of $\mathscr{E}$ and splitting of the associated real topological vector bundle $\mathscr{E}(\real)$ are sufficient to split off a free factor of rank $1$. Our next main result implies even this is not the case.  

\begin{thmintro}[See Theorem~\ref{thm:main2}]
There exists an affine open subscheme $U$ of $\mathbb{P}_\real^3$ such that $U(\real)=\mathbb{P}^3(\real)$ and a vector bundle $\mathscr{E}$ of rank $2$ on $U$ having the following properties:
\begin{itemize}[noitemsep,topsep=1pt]
\item $\det(\mathscr{E})$ is trivial.
\item $e(\mathscr{E})=0$.
\item $\mathscr{E}$ is (algebraically) indecomposable.
\item $\mathscr{E}(\real)\simeq \mathbb{P}^3(\real)\times \real^2$.
\end{itemize}
\end{thmintro}

These two theorems show that the relationship between algebraic and topological obstructions is surprisingly intricate.  It could be argued that the case of rank $2$ bundles on threefolds is quite special because of an inherent tension between the topology of real points and the algebro-geometric obstructions to splitting.  Indeed, if $X$ is a smooth real affine variety, then rank $2$ vector bundles on $X$ with trivial determinant are classified in motivic homotopy theory by maps to a space $BSL_2$.  In this case, the associated real vector bundle is classified by a map $X(\real) \to BSO(2) = BS^1 = K(\Z,2)$, and thus there is a single cohomological invariant governing splitting of the associated real vector bundle.  On the other hand, in the algebraic setting, the associated secondary obstruction arises from a homotopy sheaf that lies outside the stable range.  In higher dimensions, the algebraic obstruction simplifies slightly because of stability phenomena, whereas the topological splitting problem becomes more complicated owing to the presence of secondary obstructions as described by Liao.  There is, however, some evidence that the phenomenon we observe persists in higher dimension, and we give some remarks to this effect in the main body of our work.  That being said, at the end of the paper, we also establish some positive results in the following form.  


\begin{thmintro}[See Theorem~\ref{thm:norealpoints}]
Let $X$ be a smooth affine real variety of dimension $d\geq 3$ such that $X(\real)=\emptyset$. If $\Esc$ is a locally free $\Osc_X$-module of rank $d-1$, then $\Esc\simeq \Esc'\oplus \Osc_X$ if and only if $c_{d-1}(\Esc)=0$.
\end{thmintro}

More generally, the same result holds when the cohomology of the manifold $X(\Rb)$ with local coefficients vanishes in degree $\geq d-1$, e.g., when $X(\real)$ is contractible (Remark \ref{rem:cohomologicaldimensiondminus3}). 

\begin{thmintro}[See Theorem~\ref{thm:nocompactcomponent}]
Let $X$ be a smooth affine real variety of dimension $d\geq 3$ such that $X(\real)$ has no connected compact component. If $\Esc$ is a locally free $\Osc_X$-module of rank $d-1$, then $\Esc\simeq \Esc'\oplus \Osc_X$ if and only if $e(\Esc)=0$.
\end{thmintro}

Once again, the result is still true under the weaker condition that the singular cohomology group $\Hr^d(X(\real),\Z[L])$ is trivial for any topological line bundle $L/X(\real)$.

\section{Preliminaries}

\subsection{Motivic homotopy theory}
Let $k$ be a field.  Write $\Sm_k$ for the category of smooth, separated $k$-schemes.  Given $X\in\Sm_k$, we denote by $X_\Nis$ the small Nisnevich site of $X$ whose objects are the étale $X$-schemes (with all $X$-scheme morphisms as morphisms) and whose topology is the Nisnevich topology; we call sheaves on $X_\Nis$ Nisnevich sheaves on $X$. We denote by $X_{\mathrm{Zar}}$ the site whose objects are the open subsets of the topological space $X$ and call \emph{Zariski sheaves on $X$} sheaves on $X_{\mathrm{Zar}}$. Open immersions are étale, yielding a functor $X_{\mathrm{Zar}}\rightarrow X_\Nis$ which underlies a morphism $p:X_{\Nis}\rightarrow X_{\mathrm{Zar}}$ since the Nisnevich topology is finer than the Zariski topology. We denote the pushforward functor underlying $p$ by $p_*$, so given a Nisnevich sheaf $F$ of sets on $X$, $p_*F$ is simply the sheaf $U\mapsto F(U\hookrightarrow X)$ on $X_{\mathrm{Zar}}$.

We denote by $\mathcal{H}(k)$ the Morel--Voevodsky motivic or $\aone$-homotopy category.  We write $\mathcal{H}_{\bullet}(k)$ for the associated pointed version.  This category can be constructed in several ways, but nowadays it is common to realize it as the homotopy category of the $\infty$-category of motivic spaces.  The category of motivic spaces over $k$ can be identified as the subcategory of the $\infty$-category of presheaves of spaces on $\Sm_k$ that are Nisnevich excisive and $\aone$-invariant; this category is also a left Bousfield localization of the $\infty$-category of presheaves of spaces on $\Sm_k$ and we write $\mathrm{L}_{mot}$ for the associated localization functor.  

If $\mathscr{X}$ is a pointed simplicial presheaf on $\Sm_k$, then we write $\pi_i^{\aone}(\mathscr{X})$ for the Nisnevich sheafification of the presheaf $\pi_i(\mathrm{L}_{mot}\mathscr{X})$.  Regarding motivic spheres, we fix the convention that $S^{p,q} := S^p \wedge \gm{\sma q}$.  We also write $[\mathscr{X},\mathscr{Y}]_{\aone}$ for the $\aone$-homotopy classes of maps from $\mathscr{X}$ to $\mathscr{Y}$, and we will also write $\pi_{i,j}^{\aone}(\mathscr{X})$ for the Nisnevich sheafification of the presheaf $U \mapsto [S^{p,q} \wedge U_+,\mathscr{X}]_{\aone}$; when $j = 0$, the resulting sheaves coincide with $\bpi_i^{\aone}(\mathscr{X})$ as defined above.

\subsubsection{Strictly $\mathbb{A}^1$-invariant sheaves.} We assume now that $k$ is perfect. Recall that a Nisnevich sheaf $\mathbf{A}$ on $\Sm_k$ is \emph{strictly $\mathbb{A}^1$-invariant} if the cohomology presheaves $\Hr_\Nis^i(\text{--},\mathbf{A})$ are $\mathbb{A}^1$-invariant, namely the projection $X\times\Abb^1\rightarrow X$ onto the first factor induces an isomorphism $\Hr_\Nis^i(X,\mathbf{A})\cong\Hr_\Nis^i(X\times\mathbb{A}^1,\Abf)$ of abelian groups for every $i$ and every $X\in\Sm_k$.

\begin{ex}[Milnor--Witt $\Kr$-theory]
A very important $\Zb$-graded strictly $\mathbb{A}^1$-invariant sheaf is given by the \emph{unramified Milnor--Witt $\Kr$-theory sheaves} $\KMW_*$ of \cite[Chapter 3]{Morel08}. The graded ring $\KMW_*(F)$ is defined for a field $F$ by explicit generators and relations which in particular show that the Milnor $\Kr$-theory ring $\KM_*(F)$ is a quotient of $\KMW_*(F)$.
\end{ex}

We have the following fundamental theorem due to Morel (see, e.g., \cite[Theorems 2.19 and 6.10]{Bachmann24}).

\begin{thm}[Morel]\label{theo:pi_i_strictly_invariant}
Let $\mathcal{X}$ be a pointed simplicial presheaf on $\Sm_k$. Then $\pi_i^{\mathbb{A}^1}(\mathcal{X})$ is strictly $\mathbb{A}^1$-invariant for every $i\geq 2$.
\end{thm}

The full subcategory $\Ab^{\mathbb{A}^1}(k)$ of the category of abelian sheaves on $\Sm_k$ spanned by strictly $\Abb^1$-invariant sheaves is abelian and the embedding $\Ab^{\Abb^1}(k)\hookrightarrow\Ab(k)$ is exact (see in particular \cite[Corollary 6.23]{Morel08}). Thus for example, if $\Abf$ is a strictly $\Abb^1$-invariant sheaf and $n\in\Zb$, then the kernel ${}_n\Abf$ of the multiplication by $n$ homomorphism $\Abf\xrightarrow{n}\Abf$ is strictly $\Abb^1$-invariant.

\subsubsection{The Rost--Schmid complex of a strictly $\mathbb{A}^1$-invariant sheaf.} An important feature of these sheaves is that every strictly $\mathbb{A}^1$-invariant sheaf $\Abf$ admits a \emph{Rost--Schmid complex} given at a scheme $X\in\Sm_k$ of dimension $d$ by 
\[
\Cr_\RS(X,\Abf):0\rightarrow\bigoplus_{x\in X^{(0)}}\Ar(\kappa(x))\rightarrow\bigoplus_{x\in X^{(1)}}\Ar_{-1}(\kappa(x),\omega_{x/X})\rightarrow\cdots\rightarrow\bigoplus_{x\in X^{(d)}}\Ar_{-d}(\kappa(x),\omega_{x/X})\rightarrow 0
\] 
where the sum over the set $X^{(p)}$ of codimension $p$ points is in degree $p$, $\omega_{x/X}$ denotes the top exterior power of the $\kappa(x)$-dual of $\mathfrak{m}_x/\mathfrak{m}_x^2$ and $\Abf_{-n}$ is the iterated contraction of $\Abf$ (implicitly, we have extended presheaves on $\Sm_k$ to essentially smooth schemes such as extensions of finite transcendence degree of $k$ as explained in \cite{Morel08}). This contraction operation is defined by $\Abf_{-1}(X)=\mathrm{Coker}(\Abf(X)\xrightarrow{p^*}\Abf(X\times\mathbb{G}_m)$ where $p:X\times\mathbb{G}_m\rightarrow X$ denotes the projection; every contraction has a natural $\mathbb{G}_m$-action and we may therefore twist it by a $\mathbb{G}_m$-torsor, namely a line bundle, as we did in the above complex in considering the groups $\Ar_{-p}(\kappa(x),\omega_{x/X})$. The Nisnevich (resp. Zariski) sheafification of the functor $U\mapsto\Cr_\RS(U,\Abf)^p$\footnote{The Rost--Schmid complex is by construction functorial in smooth morphisms of schemes \cite[Paragraph 6.1.2]{Bachmann24}. In particular, since morphisms of étale $X$-schemes are étale hence smooth, as are open immersions, we do obtain a functor on $X_\Nis$ (resp. $X_{\mathrm{Zar}}$).}. defines a flasque Nisnevich (resp. Zariski) sheaf on $X$ and the Rost--Schmid complex determines a resolution of the Nisnevich (resp. Zariski) sheaf $\Abf$ (resp. $p_*\Abf$) on $X$ by these flasque sheaves. In particular, $\Hr_{\mathrm{Nis}}^i(X,\Abf)\cong\Hr_{\mathrm{Zar}}^i(X,p_*\Abf)$ is the $i$-th cohomology group of the complex $\Cr_\RS(X,\Abf)$. Since we mostly manipulate the cohomology of strictly $\mathbb{A}^1$-invariant sheaves, we generally do not indicate the topology used: by the previous argument, it is equivalent to take the Zariski or Nisnevich topology.

\begin{ex}
We have $(\KMW_n)_{-1}=\KMW_{n-1}$ and $(\KM_n)_{-1}=\KM_{n-1}$ for every $n\in\Zb$. Therefore the Rost--Schmid complex for \emph{e.g.} Milnor $\Kr$-theory at a scheme $X\in\Sm_k$ is of the form \[\Cr_\RS(X,\KM_n):0\rightarrow\bigoplus_{x\in X^{(0)}}\KM_n(\kappa(x))\rightarrow\cdots\rightarrow\bigoplus_{x\in X^{(d)}}\KM_{n-d}(\kappa(x))\rightarrow 0\] (twists are irrelevant for Milnor $\Kr$-theory since the $\mathbb{G}_m$-action on it is trivial). It coincides with the Rost complex for Milnor $\Kr$-theory constructed in \cite{Rost96}.
\end{ex}

The contraction functor $\Abf\mapsto\Abf_{-1}$ is exact \cite[Lemma 7.33]{Morel08} as is the functor $\Abf\mapsto\Abf(K)$ taking sections over $K$ for any field $K$. Consequently, the functor $\Abf\mapsto\Cr_\RS(U,\Abf)$ from the category of strictly $\Abf^1$-invariant sheaves to the category of cochain complexes is exact for every $U\in\Sm_k$. 

\subsubsection{The twisted Rost--Schmid complex.} Let $\Abf$ be a strictly $\mathbb{A}^1$-invariant sheaf, let $X\in\Sm_k$ and let $\Lc$ be a locally free $\Osc_X$-module of rank $1$. Then we have an $\Lc$-twisted Rost--Schmid complex \[\Cr_\RS(X,\Lc,\Abf_{-1}):0\rightarrow\bigoplus_{x\in X^{(0)}}\Ar_{-1}(\kappa(x),\Lc)\rightarrow\cdots\rightarrow\bigoplus_{x\in X^{(d)}}\Ar_{-d-1}(\kappa(x),\omega_{x/X}\otimes\Lc)\rightarrow 0\] for the contraction $\Mbf=\Abf_{-1}$. In particular, for every $X$-scheme $Y\rightarrow X$ with $Y\in\Sm_k$, we have a twisted complex $\Cr_\RS(Y,\Lc_{|Y},\Abf_{-1})$. We may consider the Nisnevich (resp. Zariski) sheaf $\Mbf_\Nis(\Lc)$ (resp. $\Mbf_{\mathrm{Zar}}(\Lc)$) on $X$ associated with the presheaf $(U\rightarrow X)\mapsto\Abf_{-1}(U)\otimes_{\Zb[\mathbb{G}_m(U)]}\Zb[\Lc^0(U)]$ where $\Lc^0$ is the $\mathbb{G}_m$-torsor of invertible sections of $\Lc$ on $U$, using the natural $\mathbb{G}_m$-action on the contraction to form the tensor product. Then the Nisnevich (resp. Zariski) sheafification of the functor $U\mapsto\Cr_\RS(U,\Lc_{|U},\Abf_{-1})$ on $X_\Nis$ (resp. $X_{\mathrm{Zar}}$) is a flasque sheaf and the twisted Rost--Schmid complex is a resolution of $\Mbf_\Nis(\Lc)$ (resp. $\Mbf_{\mathrm{Zar}}(\Lc)$) by these flasque sheaves. Consequently $\Hr_\Nis^p(X,\Mbf_\Nis(\Lc))\cong\Hr_{\mathrm{Zar}}^p(X,\Mbf_{\mathrm{Zar}}(\Lc))$ is the $p$-th cohomology group of the complex $\Cr_\RS(X,\Lc,\Abf_{-1})$. Moreover, we note that $\Mbf_{\Zar}(\Lc)=p_*\Mbf_\Nis(\Lc)$ (see \emph{e.g.} \cite[Lemma 2.30]{Hornbostel21}) so there is no risk of confusion in dropping the subscripts $\Nis$ and $\Zar$ which we do from now on.

\begin{ex}
The $n$-th Chow--Witt group of $X$ twisted by $\Lc$ is $\widetilde{\CH}^n(X,\Lc)=\Hr^n(X,\KMW_n(\Lc))$.
\end{ex}

By convention, from now on, a \emph{sheaf on $X$} is a sheaf on the site $X_{\mathrm{Zar}}$ associated with the topological space $X$. As mentioned previously, we essentially only consider cohomology with coefficients in strictly $\mathbb{A}^1$-invariant sheaves for which it follows from the previous arguments that the Zariski cohomology groups in fact agree with the Nisnevich cohomology groups, both in the twisted and untwisted case.

\subsection{Grothendieck--Witt groups}\label{subsection:GW-groups}
Let $X$ be a noetherian separated scheme on which $2$ is invertible\footnote{The modern technology of Poincaré $\infty$-categories allows the development of Grothendieck--Witt theory obviating the need for additional hypotheses on the characteristic of the base field.  Since our eventual goal is to study varieties over the real numbers, and since we routinely quote \cite{Schlichting17}, we leave this hypothesis here for convenience.} and let $\Lc$ be a locally free sheaf of rank $1$ on $X$. We denote by $\Ch^b(X)$ the category of bounded complexes of locally free coherent $\Osc_X$-modules, also called strictly perfect complexes. The complex having $\Lc$ concentrated in degree $0$, which we abusively denote by $\Lc$ as well, is an object of $\Ch^b(X)$ and defines a duality functor $\sharp_\Lc:\Ch^b(X)\rightarrow\Ch^b(X)$ via the internal hom in the category of coherent locally free sheaves, with biduality isomorphism $\mathrm{can}_E^\Lc:E\rightarrow E^{\sharp_\Lc\sharp_\Lc}$ given by $\mathrm{can}_E^\Lc(x)(f)=(-1)^{|x||f|}f(x)$ (where $|y|$ is the degree of the homogeneous element $y$ of the graded module underlying an object of $\Ch^b(\Esc)$ if $y\neq 0$, and $0$ otherwise). Therefore denoting by $\mathrm{quis}$ the collection of quasi-isomorphisms of objects of $\Ch^b(X)$, we obtain a quadruple $(\Ch^b(X),\mathrm{quis},\sharp_\Lc,\mathrm{can}^\Lc)$ which is a \emph{dg-category with weak equivalences and duality} in the sense of \cite{Schlichting17}. To such a structure, Schlichting attaches for each $n$ a spectrum $\GWc^{n}(X,\Lc)$ which is the $n$-th shifted Grothendieck--Witt spectrum of $X$ with coefficients in $\Lc$.

\begin{defn}
The $n$-shifted, degree $i$ Grothendieck--Witt group of $X$ twisted by the line bundle $\Lc$ is the group
\[
\GWr_i^{n}(X,\Lc) := \pi_i(\GWc^{n}(X,\Lc)).
\]
\end{defn}

The sheaf $\Lc$ (resp. the integer $n$) will be omitted from the notation when it is the structure sheaf $\Osc_X$ (resp. the integer $0$) yielding spectra denoted $\GWc(X,\Lc)$ and $\GWc^{n}(X)$. Note also that $\GWc^{n}(X,\Lc)$ is $4$-periodic in $n$ (as an object of the homotopy category of spectra) \cite[Remark 5.9]{Schlichting17} and only depends on the class of $\Lc$ in $\Pic(X)/2$.

We simplify the notation slightly by setting $\GWr^n(X,\Lc)=\GWr_0^n(X,\Lc)$; this group coincides (up to natural isomorphism) with the $n$-shifted Grothendieck--Witt group of $X$ twisted by $\Lc$ in the sense of Walter by \cite[Corollary 3.13, Remark 3.14, Proposition 5.6]{Schlichting17}.

\begin{ex}
The group $\GWr^0(X,\Lc)$ (resp. $\GWr^2(X,\Lc)$) is the Grothendieck--Witt group of symmetric (resp. skew-symmetric) forms on $X$ with coefficients in $\Lc$ by \cite[Theorem 6.1]{Walter03}. Moreover, if $i<0$, then $\GWr_i^{n}(X,\Lc)$ is naturally isomorphic to the shifted Witt group $\Wr^{n-i}(X,\Lc)$ in the sense of \cite{Balmer00} by \cite[Example after Proposition 5.6]{Schlichting17}.
\end{ex}

From now on, we assume that $X$ is a smooth algebraic variety over a field $k$ of characteristic not $2$ (the reader is free to assume that $k$ is an extension of $\Rb$). The category $\Ch^b(X)$ then has a $\Kr$-theory spectrum $\mathcal{K}(X)$ whose homotopy groups are the higher algebraic $\Kr$-theory groups of $X$, yielding 
in particular $\pi_0(\mathcal{K}(X))=\Kr_0(X)$. The \emph{algebraic Bott sequence} of \cite[Theorem 6.1]{Schlichting17} is then an exact triangle in the homotopy category of spectra of the form \[\GWc^{n}(X,\Lc)\xrightarrow{F}\mathcal{K}(X)\xrightarrow{H}\GWc^{n+1}(X,\Lc)\xrightarrow{\eta\cup}\mathrm{S}^1\wedge\GWc^{n}(X,\Lc)\] where $\eta$ is an explicit element of $\GWr_{-1}^{-1}(k)$ (corresponding to $\langle 1\rangle$ under the natural isomorphism $\GWr_{-1}^{-1}(k)\simeq \mathrm{W}(k)$). Here $F$ is the forgetful homomorphism, forgetting the duality structure, and $H$ is the hyperbolic homomorphism. The long exact sequence of homotopy groups associated with the algebraic Bott sequence yields an exact sequence \[\GWr^{n}(X,\Lc)\xrightarrow{F}\Kr_0(X)\xrightarrow{H}\GWr^{n+1}(X,\Lc)\rightarrow\Wr^{n+1}(X,\Lc)\rightarrow\Kr_{-1}(X)=0\] of abelian groups which we refer to as Karoubi periodicity.

If $f:Y\rightarrow X$ is a morphism of schemes and $\Lc$ is a line bundle on $X$, there is an induced morphism $f^*:(\Ch^b(X),\mathrm{quis},\sharp_\Lc,\mathrm{can}^\Lc)\rightarrow(\Ch^b(Y),\mathrm{quis},\sharp_{f^*\Lc},\mathrm{can}^{f^*\Lc})$ of dg-categories with weak equivalences and duality, yielding a morphism $f^*:\GWc^j(X,\Lc)\rightarrow\GWc^j(Y,f^*\Lc)$ of spectra that can be arranged to be a functor of $f$ \cite[Remark 9.4]{Schlichting17}. Therefore we get an induced homomorphism \[f^*:\GWr_i^j(X,\Lc)\rightarrow\GWr_i^j(Y,f^*\Lc)\] of abelian groups for every $(i,j)$. In particular, we have a presheaf $U\mapsto\GWr_i^j(U)$ of abelian groups on $\Sm_k$ for every pair $(i,j)$ of integers. We denote the associated Nisnevich sheaf on $\Sm_k$ by $\GW_i^j$. The sheaf $\GW_i^j$ is also the $i$-th homotopy sheaf $\pi_i^{\mathbb{A}^1}(\underline{\mathcal{GW}}^j)$ of the motivic space $\underline{\mathcal{GW}}^j$ representing $j$-shifted Grothendieck--Witt groups constructed by Hornbostel in \cite{Hornbostel05}, see also \cite[\S 3.3]{Asok12c}. Consequently, $\GW_i^j$ is strictly $\mathbb{A}^1$-invariant by Morel's theorem \ref{theo:pi_i_strictly_invariant}.

Moreover, if $f$ is a proper morphism of constant relative dimension $\delta=\dim(Y)-\dim(X)$ where $X$ and $Y$ are smooth over $k$, we have a pushforward homomorphism \[f_*:\GWr^{i}(Y,\omega_{Y/k}\otimes f^*\Lc)\rightarrow\GWr^{i-\delta}(X,\omega_{X/k}\otimes\Lc)\] for each line bundle $\Lc$ on $X$, where $\omega_{Z/k}$ is the determinant of the sheaf of Kähler differentials of the smooth $k$-scheme $Z$. For example, if $X$ is irreducible of dimension $d$, every rational point $x:\Spec k\rightarrow X$ of $X$ induces a homomorphism \[x_*:\GWr^{i}(k)\rightarrow\GWr^{i+d}(X,\omega_{X/k})\] of abelian groups for every $i$.

\begin{ex}\label{exe:pushfoward_GW_even_degree}
Using the Karoubi periodicity exact sequence, one can compute some Grothendieck--Witt groups of fields, as in \cite[Lemma 4.1]{Fasel08c}. These computations show that $\GWr^2(k)=\Zb$ and $\GWr^3(k)=\Zb/2$ if $k$ is a field of characteristic not $2$. If $k'/k$ is an even degree extension, the finite map $p:\Spec k'\rightarrow\Spec k$ induces a pushforward homomorphism $p_*:\GWr^3(k')\rightarrow\GWr^3(k)$, which we claim to be trivial. Indeed consider the commutative diagram
\begin{center}
\begin{tikzcd}
\Kr_0(k') \arrow[d,"p_*"] \arrow[r,"H"] & \GWr^3(k')=\Zb/2 \ar[r] \arrow[d,"p_*"] & 0 \\
\Kr_0(k) \arrow[r,"H"] & \GWr^3(k)=\Zb/2 \arrow[r]                                & 0
\end{tikzcd}
\end{center}
with exact rows. To show that $p_*:\GWr^3(k')\rightarrow\GWr^3(k)$ is the zero map, it suffices to observe that $p_*:\Kr_0(k')\rightarrow\Kr_0(k)$ is multiplication by $[k':k]$.
\end{ex}

One of the most important tools for studying Grothendieck--Witt groups on a smooth $k$-scheme $X$ is the \emph{Gersten--Grothendieck--Witt spectral sequence}, see for instance \cite{Fasel09c}. Given a line bundle $\Lc$ on $X$, it is a cohomological spectral sequence of the form \[\Er(n)_1^{p,q}=\bigoplus_{x\in X^{(p)}}\GWr_{n-p-q}^{n-p}(\kappa(x),\omega_{x/X}\otimes\Lc)\Rightarrow\GWr_{n-m}^n(X,\Lc)\] where the filtration on the abutment is given by the codimension of the support as usual in coniveau spectral sequences (the groups outside the range $0\leq p\leq\dim(X)$ vanish). The line $q=0$ is the complex 
\[
\Er(n)_1^{p,0}:0\rightarrow\bigoplus_{x\in X^{(0)}}\GWr_n^n(\kappa(x),\omega_{x/X}\otimes\Lc)\rightarrow\cdots\rightarrow\bigoplus_{x\in X^{(d)}}\GWr_{n-d}^{n-d}(\kappa(x),\omega_{x/X}\otimes\Lc)\rightarrow 0
\] 
where $d=\dim X$. As explained in \cite[Subsections 3.3 and 3.4]{Fasel09c}, this complex is quasi-isomorphic to the Rost--Schmid complex for $\KMW_n(\Lc)$ in degree $p\geq n-2$.
In particular, $\Er(n)_2^{n,0}=\widetilde{\CH}^n(X,\Lc)$ \cite[Theorem 33]{Fasel09c}. More generally, the line $q$ is the Rost--Schmid complex for the strictly $\mathbb{A}^1$-invariant sheaf $\GW_{n-q}^n$.

\begin{ex}\label{exe:GWWss_dim_3}
Let us analyse the spectral sequence $\Er(2)$ over a threefold $X\in\Sm_k$, and more specifically the line $p+q=2$ which will be relevant below. Let us begin by noticing that 
\[
\Er(2)_1^{0,2}=\bigoplus_{x\in X^{(0)}}\GWr_0^2(\kappa(x),\omega_{x/X}\otimes\Lc)=\bigoplus_{x\in X^{(0)}}\Z
\]
by Example \ref{exe:pushfoward_GW_even_degree}. The contractions of $\Z$ being trivial, we see that $\Er(2)_1^{i,2}=0$ for any $i\geq 1$. Moreover, the edge homomorphism 
\[
\GWr^2(X,\Lc)\to \Er(2)_{\infty}^{0,2}
\]
coincides with the localization homomorphism $\GWr^2(X,\Lc)\to \GWr^2(k(X),\Lc\otimes k(X))=\Z$, which coincides with the rank homomorphism, whose kernel we denote by $\widetilde{\GWr}^2(X,\Lc)$.
Next, we claim that
\[
\Er(2)_1^{0,1}=\bigoplus_{x\in X^{(0)}}\GWr_1^2(\kappa(x),\omega_{x/X}\otimes\Lc)
\] 
is the zero group. Indeed, this follows from the fact that $\GWr_1^2(F)$ vanishes for any field $F$ of characteristic not $2$ by \cite[Lemma 2.2]{Fasel10b}. It follows that the subquotients $\Er(2)_2^{0,1}$ and $\Er(2)_3^{0,1}$ of $\Er(2)_1^{0,1}$ vanish, while the groups $\Er(2)_1^{i,1}=0$ (being contractions of a trivial sheaf). Therefore the abutment group $\Er(2)_\infty^{1,1}$ vanishes.

We have $\Er(2)_2^{2,0}=\widetilde{\CH}^2(X,\Lc)$. Since the terms of the spectral sequence vanish outside the range $0\leq p\leq\dim(X)=3$, the differentials out of $\Er(2)_r^{2,0}$ vanish for every $r\geq 2$, and the only differential into $\Er(2)_r^{2,0}$ that may not vanish is the differential $d_2^{0,1}:\Er(2)_2^{0,1}\rightarrow\Er(2)_2^{2,0}$. However, $\Er(2)_2^{0,1}$ vanishes as observed previously. Therefore $\widetilde{\CH}^2(X,\Lc)$ coincides with the group $\Er(2)_\infty^{2,0}$ at the $\Er_\infty$-page.

The group $\Er(2)_2^{3,-1}$ is given by $\Er(2)_2^{3,-1}=\Hr^3(X,\GW_3^2(\Lc))$ by our observation regarding the lines $q=q_0<0$ above. There is a differential \[d_2^{1,0}:\Er(2)_2^{1,0}\rightarrow\Er(2)_2^{3,-1}\] on the $\Er(2)_2$-page. Our partial identification of the Gersten--Grothendieck--Witt complex with the Rost--Schmid complex of $\KMW_2(\Lc)$ shows that $\Er(2)_2^{1,0}=\Hr^1(X,\KMW_2(\Lc))$. Since the differential $d_2^{3,-1}:\Er(2)_2^{3,-1}\rightarrow\Er(2)_2^{5,-2}$ vanishes, we obtain that $\Er(2)_3^{3,-1}$ is a quotient group of the form \[\Er(2)_3^{3,-1}=\Hr^3(X,\GW_3^2(\Lc))/d_2^{1,0}\Hr^1(X,\KMW_2(\Lc)).\] The differential $d_3^{0,1}:\Er(2)_3^{0,1}\rightarrow\Er(2)_3^{3,-1}$ vanishes because the source $\Er(2)_3^{0,1}$ vanishes as we observed previously, and the target of the differential $d_3^{3,-1}$ is again the zero group because the terms on the $\Er(2)_1$-page vanish outside the range $0\leq p\leq 3$. Finally, the differentials to and from $\Er(2)_r^{3,-1}$ vanish for $r\geq 4$ for the same reason. In conclusion, $\Er(2)_\infty^{3,-1}=\Hr^3(X,\GW_3^2(\Lc))/d_2^{1,0}\Hr^1(X,\KMW_2(\Lc))$.

Codimension of support yields a filtration $\Fr^\bullet\GWr^2(X,\Lc)$ on the abutment group $\GWr^2(X,\Lc)$ on the line $p+q=2$ such that the graded piece 
\[
\Gr^p\GWr^2(X,\Lc)=\Fr^p\GWr^2(X,\Lc)/\Fr^{p+1}\GWr^2(X,\Lc)
\] 
is equal to $\Er(2)_\infty^{p,2-p}$. The previous arguments then show that 
\[
\Gr^p\GWr^2(X,\Lc)=0\;\text{if}\;p<0\;\text{or}\;p>d,\;\Gr^0=\mathrm{H}^0(X,\Z),\;\Gr^1\GWr^2(X,\Lc)=\Er(2)_\infty^{1,1}=0,
\]
\[
\Gr^2\GWr^2(X,\Lc)=\widetilde{\CH}^2(X,\Lc),\;\Gr^3\GWr^2(X,\Lc)=\Hr^3(X,\GW_3^2(\Lc))/d_2^{1,0}\Hr^1(X,\KMW_2(\Lc))
\] 
In particular, this yields an exact sequence 
\begin{equation}\label{eqn:GGW2ss}
\Hr^1(X,\KMW_2(\Lc))\xrightarrow{d_2^{1,0}}\Hr^3(X,\GW_3^2(\Lc))\rightarrow\widetilde{\GWr}^2(X,\Lc)\rightarrow\widetilde{\CH}^2(X,\Lc)\rightarrow 0.
\end{equation}
\end{ex}

\subsection{Vector bundles}\label{sec:motivicsplitting}
In analogy with the situation in topology, there is a well-behaved obstruction theory for splitting vector bundles in motivic homotopy theory. We briefly survey the discussion in \cite[\S 6]{Asok12c} for the convenience of the reader.  If $X\in \Sm_k$, then write $\mathscr{V}_r(X)$ for the set of isomorphism classes of rank $r$ algebraic vector bundles on $X$.  In that case, paralleling the Pontryagin--Steenrod classification of vector bundles in topology, there is a corresponding affine representability statement: if $X$ is furthermore affine, then there is a canonical pointed bijection:
\[
\mathcal{V}_n(X)\simeq [X_+,\mathrm{BGL}_n]_{\mathbb{A}^1,\bullet}
\]
where $\mathrm{BGL}_r$ is the classifying space of the general linear group in motivic spaces (and, up to motivic equivalence, the map $Gr_r \to \mathrm{BGL}_r$ classifying the tautological vector bundle is a motivic equivalence \cite[\S 4 Proposition 3.7]{Morel99}); see \cite{Morel08,Schlichting15} and \cite[Theorem 1]{Asok15a}. The determinant map $\mathrm{GL}_n\to \gm{}$ yields a morphism of classifying spaces $\det\colon\mathrm{BGL}_n\to \mathrm{B}\gm{}$ such that the induced map
\[
\mathcal{V}_n(X)\simeq [X_+,\mathrm{BGL}_n]_{\mathbb{A}^1,\bullet}\to [X_+,\mathrm{B}\gm{}]_{\mathbb{A}^1,\bullet}\simeq \mathrm{Pic}(X)
\]
is the determinant map. If $\Lc\in \mathrm{Pic}(X)$, we denote by $\mathcal{V}_n(X,\Lc)$ the fiber of $L$ under this morphism. Equivalently, an element in $\mathcal{V}_n(X,\Lc)$ is an isomorphism class of pairs $(\mathscr{E},\varphi)$ where $\mathscr{E}$ is a rank $n$ vector bundle on $X$ and $\varphi\colon\det(\mathscr{E})\to \Lc$ is a specified isomorphism, and two pairs $(\mathscr{E},\varphi)$ and $(\mathscr{E}',\varphi')$ are isomorphic if there exists an isomorphism $\psi\colon \mathscr{E}\to \mathscr{E}'$ such that $\varphi=\varphi'\circ\det(\psi)$.

If $n\geq 2$, the map $\mathrm{GL}_{n-1}\to \mathrm{GL}_{n}$ defined by $M\mapsto \mathrm{diag}(1,M)$ induces a morphism $s_{n-1}\colon \mathrm{BGL}_{n-1}\to \mathrm{BGL}_{n}$ of pointed spaces. If $\mathscr{E}$ is a rank $n$ vector bundle on $X$, represented by a pointed morphism $\xi\colon X_+ \to \mathrm{BGL}_{n}$, there is a set of inductively defined obstructions to lifting $\xi$ to a morphism $\xi'\colon  X_+ \to \mathrm{BGL}_{n-1}$ (having the property that $\xi=s_n\circ \xi'$ up to homotopy).  The obstructions arise via pullbacks of $k$-invariants in the Moore-Postnikov tower of the morphism $s_n$.  These obstructions take values in cohomology of $X$ with coefficients in $\aone$-homotopy sheaves of the (homotopy) fiber of $s_n$, twisted by a suitable orientation local system.   

Connectivity estimates on the homotopy fiber yield a well-defined \emph{primary} obstruction to lift $\xi$ as above, which lives in  the cohomology group $\mathrm{H}^n_{\Nis}(X,\piaone_{n-1}(F)(\det \mathscr{E}))$; here $F$ is the homotopy fiber of $s_n$, which is $\aone$-equivalent to $\A^n\smallsetminus\{0\}$, and the strictly $\A^1$-invariant sheaf $\piaone_{n-1}(F)$ is twisted by the line bundle $\det V$, which plays the role of an orientation character. This primary obstruction is the so-called \emph{Euler class}
\[
e(V)\in \mathrm{H}^n_{\Nis}(X,\piaone_{n-1}(\A^n\smallsetminus\{0\})(\det \mathscr{E}))=\mathrm{H}^n_{\Nis}(X,\KMW_n(\det \mathscr{E}))=\CHW^n(X,\det \mathscr{E}).
\]
If $n=d:=\mathrm{dim}(X)\geq 2$, then the Euler class is the only obstruction to lifting $\xi$, and we obtain for any line bundle $\Lc$ on $X$ an exact sequence of pointed sets
\begin{equation}\label{eqn:corank0}
\mathcal{V}_{d-1}(X,\Lc)\xrightarrow{s_{d-1}}\mathcal{V}_{d}(X,\Lc)\xrightarrow{e}\CHW^n(X,\Lc)
\end{equation}
with $e$ defined by $e(\mathscr{E},\varphi)=\varphi_*e(\mathscr{E})$, where $\varphi_*\colon \CHW^n(X,\det \mathscr{E})\to \CHW^n(X,\Lc)$ is the isomorphism induced by $\varphi$.

If $n=\mathrm{dim}(X)-1\geq 2$, the vanishing of the primary obstruction yields a well-defined \emph{secondary} obstruction $o_2(V)$ that lives in the cokernel of a morphism
\[
\partial\colon \mathrm{H}^{d-2}_{\Nis}(X,\piaone_{d-2}(\A^{d-1}\smallsetminus\{0\})(\det \mathscr{E}))\longrightarrow \mathrm{H}^{d}_{\Nis}(X,\piaone_{d-1}(\A^{d-1}\smallsetminus\{0\})(\det \mathscr{E}))
\]
provided by the Moore-Postnikov factorization.  If this secondary obstruction vanishes, then there exists a further lift of $\xi$ to a higher stage of the tower.  More precisely, the Moore-Postnikov tower provides us with a space $E\in \mathcal{H}_\bullet(k)$ over $\mathrm{B}\gm{}$ sitting in a sequence
\[
\mathrm{BGL}_{d-2}\xrightarrow{i}E\xrightarrow{p}\mathrm{BGL}_{d-1}
\]
yielding for any line bundle $\Lc$ on $X$ a diagram of pointed sets and groups
\begin{equation}
\xymatrix{&\mathcal{V}_{d-2}(X,\Lc)\ar[d]^-{i_*}\ar[rd]^-{(s_{d-2})_*} & &  \\
\mathrm{H}^{d-2}_{\Nis}(X,\KMW_{d-1}(\Lc))\ar[r]\ar[rd]_-\partial & [X_+,E^{\Lc}]_{\mathbb{A}^1,\bullet}\ar[r]_-{p_*}\ar[d]^-{o_2} & \mathcal{V}_{d-1}(X,\Lc)\ar[r]^-e & \CHW^{d-1}(X,\Lc)  \\
 & \mathrm{H}^{d}_{\Nis}(X,\piaone_{d-1}(\A^{d-1}\smallsetminus\{0\})(\Lc)) &  & }
\end{equation}
in which $E^{\Lc}$ fits into a homotopy Cartesian square
\[
\xymatrix{E^{\Lc}\ar[r]\ar[d] & E\ar[d] \\
X\ar[r]_-{\det\circ\xi} & \mathrm{B}\gm{},}
\]
the pointed set $[X_+,E^{\Lc}]_{\mathbb{A}^1,\bullet}$ is computed as homotopy classes over $X_+$, the vertical line is exact in the usual sense, and the horizontal line is exact in the sense that two elements of $[X_+,E^{\Lc}]_{\mathbb{A}^1,\bullet}$ have the same image under $p_*$ if and only if they differ by an element of  the group $\mathrm{H}^{d-2}_{\Nis}(X,\KMW_{d-1}(\Lc))$ which acts naturally on this set. We note that we will use the above diagram only in the case where $\Lc$ is trivial, in which case one can remove the decoration $\Lc$ and just consider the pointed set $[X_+,E]_{\mathbb{A}^1,\bullet}$ (computed over the point).

In order to use this diagram, it is necessary to understand the homotopy sheaf $\piaone_{d-1}(\A^{d-1}\smallsetminus\{0\})$, together with the induced action of $\gm{}=\piaone_1(\mathrm{BGL}_n)$ that determines the associated twisted sheaf. In this article, we will mainly consider vector bundles with a given trivialization of their determinants, in which case this action is irrelevant. As regards the relevant homotopy sheaf, it can be described in case $n\geq 4$ by an exact sequence of sheaves of the form (\cite[\S 4.3]{Asok22}, the sequence is further exact on the left if $k$ has characteristic $0$ \cite[Theorem 7.2.2.3]{asokP1stabilizationUnstableMotivicToappear})
\begin{equation}\label{eqn:firstnontrivialhomotopysheaf}
\KM_{n+2}/24 \longrightarrow \piaone_n(\A^n\smallsetminus\{0\}) \longrightarrow \mathbf{GW}_{n+1}^n.
\end{equation}
The cases $n=2,3$ were treated earlier in \cite{Asok12a} and \cite{Asok12c} respectively; the descriptions of these sheaves also involve the degree map $\piaone_n(\A^n\smallsetminus\{0\})\to \mathbf{GW}_{n+1}^n$, which can be seen as an unstable analogue of the unit map $\epsilon\colon \mathbf{1}\to \mathbf{KO}$ from the sphere spectrum to the $\pone$-spectrum representing Hermitian $K$-theory (\cite{Asok14b}). 

\subsection{Real realization and its avatars}\label{sec:realrealization}
Suppose that the base field $k$ admits a fixed real embedding $k\hookrightarrow \real$. Left Kan extension of the functor that assigns to a smooth $k$-variety $X$ the real manifold $X(\real)$ (endowed with the usual Euclidean topology) extends to a real realization functor
\[
{\mathfrak R}_{{\mathbb R}}: \mathrm{H}_\bullet(k) \longrightarrow \mathrm{H}_\bullet
\]
where $\mathrm{H}_\bullet$ is the ``classical'' homotopy category of spaces.  As a left Kan extension, this functor preserves (homotopy) colimits, as well as finite products.  In particular, the real realization of the motivic sphere $S^{p,q}$ is the ordinary sphere $S^p$, and there are induced homomorphisms: 
\[
\bpi_{i,j}^{\aone}(X,x)(\real) \longrightarrow \pi_{i}(X(\real),x).
\]
Many of the fundamental problems we need to address can be phrased as problems of commuting (homotopy) limits and colimits.  Indeed, we will routinely use fiber sequences when studying obstruction theory, and there is no formal reason for fiber sequences to be preserved by real realization.  This leads to what might be interpreted as unexpected behavior of some spaces, e.g., the real realizations of Voevodsky's motivic Eilenberg--Mac Lane spaces are analyzed in \cite[\S 7.3]{ABEH}.

Nevertheless, when fiber sequences may be realized geometrically, such compatibilities are straightforward to check.  For example, using the fact that the simplicial classifying space $BGL_n$ can be described as a geometric realization, for $n \geq 2$, the motivic or $\aone$-fiber sequence 
\[
\A^n\smallsetminus\{0\} \longrightarrow \mathrm{BGL}_{n-1} \longrightarrow \mathrm{BGL}_{n}
\]
realizes to the fiber sequence
\[
S^{n-1}\simeq \real^n\smallsetminus\{0\} \longrightarrow \mathrm{BO}_{n-1}(\real) \longrightarrow \mathrm{BO}_{n}(\real).
\]

Our general goal is to compare the Moore-Postnikov towers associated to these fiber sequences, and we'll need concrete realization morphisms in order to perform this comparison. We will introduce the real cycle class map in the next section, but we begin by recalling a well-known classical result.  



\begin{lem}\label{lem:top_splitting_over_threefolds}
Let $E$ be a rank $2$ topological real vector bundle on a three-dimensional manifold $M$ such that the Euler class of $E$ vanishes. Then $E$ splits as $E\simeq\det E\oplus\varepsilon^1$ where $\varepsilon^1$ is the trivial rank $1$ bundle on $M$ and $\det E$ is the determinant bundle of $E$. In particular, if $\det E$ is trivial, then $E$ is trivial.
\end{lem}



\subsubsection{Real cycle class maps.}\label{subsection:real_cycle_class_maps}
Let $X$ be a smooth real algebraic variety of dimension $d$. Denote by $\iota:X(\real)\hookrightarrow X$ the inclusion of the real locus of $X$: here $X(\real)$ is again endowed with its Euclidean topology, for which $\iota$ is continuous. The map $\iota$ induces a pushforward functor $\iota_*$ on categories of abelian sheaves. Recall the following result (e.g. \cite[Lemma 1.2]{vanHamel00}).

\begin{prop}\label{prop:cohomology_locally_constant_sheaves}
If $\mathscr{F}$ is a locally constant sheaf on $X(\Rb)$, then $\mathrm{R}^p\iota_*\Fsc=0$ for every $p>0$ (where $\mathrm{R}^p\iota_*$ is the $p$-th right derived functor of $\iota_*$). Thus $\Hr^n(X,\iota_*\Fsc)=\Hr^n(X(\Rb),\Fsc)$.
\end{prop}

We now define the real cycle class maps that we will use. We start with the most classical one. Denote by $\mathscr{H}^n$ the sheaf on $X$ associated with the presheaf $U\mapsto\Hr_\et^n(U,\Z/2)$. Then there is a natural morphism $\overline{\mathrm{sign}}:\mathscr{H}^n\rightarrow\iota_*\Z/2$ of abelian sheaves on $X$ constructed in \cite[Subsection 2.1]{Colliot90} where it is called the signature mod $2$, yielding for any $(n,t)$ a morphism 
\[
\overline{\gamma_\real}:\Hr^n(X,\mathscr{H}^t)\rightarrow\Hr^n(X,\iota_*\Z/2)
\] 
and the last group may naturally be identified with the sheaf cohomology group $\Hr^n(X(\real),\Z/2)$ (which coincides with the corresponding singular cohomology group) by Proposition \ref{prop:cohomology_locally_constant_sheaves}. Symbols in étale cohomology may be used to define a morphism $\KM_t\rightarrow\mathscr{H}^t$ of sheaves on $X$ for any $t$ which factors through $\KM_t/2$ since the target sheaf is $2$-torsion by definition. Consequently, we have morphisms $\overline{\mathrm{sign}}:\KM_t\rightarrow\iota_*\Z/2$ and $\overline{\mathrm{sign}}:\KM_t/2\rightarrow\iota_*\Z/2$ of abelian sheaves on $X$ and thus morphisms 
\[
\overline{\gamma_\real}:\Hr^n(X,\KM_t)\rightarrow\Hr^n(X(\real),\Z/2),\;\overline{\gamma_\real}:\Hr^n(X,\KM_t/2)\rightarrow\Hr^n(X(\real),\Z/2)
\] 
of abelian groups for any $(n,t)$. We call these \emph{mod $2$ cycle class maps}.

There is no known way to lift the first of these maps to \emph{integral} singular cohomology. To map an algebraic cohomology theory into integral singular cohomology, following \cite{Jacobson16}, we use the theory of symmetric bilinear forms. Namely let $\Ibf^n$ denote the sheafification of the presheaf $U\mapsto\mathrm{I}^n(U)$ on $X$, where $\mathrm{I}^n(U)$ is the $n$-th power of the fundamental ideal $\mathrm{I}(U)$ of the Witt ring $\mathrm{W}(U)$ of $U$. Then the signature of symmetric bilinear forms over $\real$ allows one to define a morphism $\mathrm{sign}_t:\Ibf^t\rightarrow\iota_*\Z$ of sheaves which we take to be the composition of the morphism of \cite[Page 381]{Jacobson16} with multiplication by $\frac{1}{2^t}$. Multiplication by $2$ induces a morphism $f:\Ibf^t\rightarrow\Ibf^{t+1}$ such that $\mathrm{sign}_{t+1}\circ f=\mathrm{sign}_t$. As explained in \emph{e.g.} \cite[Subsection 3.2]{Hornbostel21}, this morphism can be further twisted by the datum of a line bundle $\Lc$ on $X$, yielding a morphism $\mathrm{sign}(\Lc):\Ibf^t(\Lc)\rightarrow\iota_*\Z(L)$ of sheaves where $L=\Lc(\real)$ and $\Z(L)$ is the associated local system \cite[Subsection 2.5.1]{Hornbostel21}. This provides a map 
\[
\gamma_t^n:\Hr^n(X,\Ibf^t(\Lc))\rightarrow\Hr^n(X(\real),\Z(L))
\] 
such that for every $t\leq t'$, the morphism $\Ibf^t\rightarrow\Ibf^{t'}$ induced by multiplication by $2^{t'-t}$ fits into a commutative diagram
\[
\xymatrix{\Hr^n(X,\Ibf^t(\Lc))\ar[r]\ar[rd]_-{\gamma_t^n} & \Hr^n(X,\Ibf^{t'}(\Lc))\ar[d]^-{\gamma_{t'}^n} \\
& \Hr^n(X(\real),\Z(L))}
\]
for every $n$. The map $\gamma_t^n(\Lc)$ is by definition Jacobson's \emph{real cycle class map}. The following theorem, whose non-formal content is due to Jacobson, is crucial.

\begin{thm}[Jacobson]\label{theo:jacobson}
Let $t>d$. Then $\gamma_t^n(\Lc):\Hr^n(X,\Ibf^t(\Lc))\rightarrow\Hr^n(X(\real),\Z(L))$ is an isomorphism for every $n$.
\end{thm}

Multiplication by $2$ induces an exact sequence 
\[
0\rightarrow\Z(L)\xrightarrow{\cdot 2}\Z(L)\rightarrow\Z/2\rightarrow 0
\]
of sheaves on $X(\real)$, which is preserved by $\iota_*$ by Proposition \ref{prop:cohomology_locally_constant_sheaves}. We then have a commutative diagram
\begin{equation}\label{diag:morphism_exact_sequence_real}
\xymatrix{0\ar[r] & \Ibf^{n+1}(\Lc) \ar[r] \ar[d]_-{\mathrm{sign}_{n+1}} & \Ibf^n(\Lc) \ar[r] \ar[d]_-{\mathrm{sign}_{n}} & \overline{\Ibf}^n \ar[r] \ar[d]_-{\overline{\mathrm{sign}}_{n}} & 0 \\
0 \ar[r] & \iota_*\Z(L) \ar[r]_-{\cdot 2} & \iota_*\Z(L) \ar[r] & \iota_*\Z/2 \ar[r] & 0}
\end{equation}
with exact rows by definition of the signature, where $\overline{\Ibf}^n$ is the quotient $\Ibf^n(\Lc)/\Ibf^{n+1}(\Lc)$ (this quotient sheaf does not depend on $\Lc$). Pfister forms can be used to define an epimorphism $\KM_n\rightarrow\overline{\Ibf}^n$ of sheaves which factors through $\KM_n/2$. The composite morphism $\KM_n\rightarrow\overline{\Ibf}^n\rightarrow\iota_*\Z/2$ coincides with the morphism $\overline{\mathrm{sign}}$ defined above \cite[Proof of Proposition 3.12]{Hornbostel21}.

Recall now that the Milnor--Witt $\mathrm{K}$-theory sheaves $\KMW_n$ are defined by Morel in \cite{Morel08}. These can also be twisted by the line bundle $\Lc$. We have a natural morphism $\KMW_n(\Lc)\rightarrow\Ibf^n(\Lc)$  (derived from \cite[Lemma 1.2]{Fasel20b}) inducing a signature morphism $\widetilde{\mathrm{sign}}(\Lc):\KMW_n(\Lc)\rightarrow\Ibf^n(\Lc)\rightarrow\iota_*\Z(\mathcal{L})$ and thus a real cycle class map 
\[
\widetilde{\gamma_\real}(\Lc):\Hr^n(X,\KMW_t(\Lc))\rightarrow\Hr^n(X(\real),\Z(L)).
\] 
For $n=t$, we obtain a morphism $\widetilde{\gamma_\real}(\Lc)\colon {\CHW}^n(X,\Lc)\rightarrow\Hr^n(X(\real),\Z(L))$. Note that the short exact sequence of sheaves
\[
0\to 2\KM_n\to \KMW_n(\Lc)\to \Ibf^n(\Lc)\to 0
\]
and the fact that $\Hr^{n+1}(X,2\KM_n)=0$ show that the morphism $\KMW_n(\Lc)\rightarrow\Ibf^n(\Lc)$ induces an epimorphism $\CHW^n(X,\Lc)\to \Hr^n(X,\Ibf^n(\Lc))$. In particular, the image of $\widetilde{\gamma_\real}(\Lc)$ is the same as the image of $\gamma_n^n(\Lc)$.

\subsubsection{Characteristic classes.}

Having our real cycle class map at hand, we will also need to understand the real realization of some characteristic classes, in particular the Borel classes introduced in \cite{Panin21} (to which we refer for more information).
Let $X$ be a smooth scheme over $\real$, and let $\mathscr{E}$ be a symplectic bundle on $X$, i.e. a vector bundle endowed with a nondegenerate symplectic form 
\[
\varphi\colon \mathscr{E}\to \mathscr{E}^\vee:=\mathcal{H}om_{\Osc_X}(\mathscr{E}, \Osc_X)
\]
To such a bundle, we can associate Borel classes $b_i(\mathscr{E},\varphi)\in \CHW^{2i}(X)$, which satisfy the usual formulas, including the Whitney sum formula. This yields maps 
\[
b_i\colon \mathrm{GW}^2(X) \longrightarrow \CHW^{2i}(X),
\]
such that $b_1$ is a group homomorphism. Compatibility between the Borel classes and the Chern classes of $\mathscr{E}$ are encoded in the commutativity of the following commutative diagram:
\[
\xymatrix{ \mathrm{GW}^2(X)\ar[r]^-{b_i}\ar[d]_-f &  \CHW^{2i}(X)\ar[d] \\
\mathrm{K}_0(X)\ar[r]_{c_{2i}} & \CH^{2i}(X),}
\]
where $f$ is the forgetful homomorphism, and the right-hand vertical arrow is the one induced by the morphism of sheaves $\KMW_*\to \KM_*$. The Borel classes are also compatible under real realization with characteristic classes in topology.  If $\mathscr{E}$ is of rank $2n$, the real realization $\mathscr{E}(\real)$ is a complex vector bundle of rank $n$ on $X(\real)$ \cite{Arnold85}, and we can consider its associated Chern classes $c_i^{\mathrm{top}}\in \Hr^{2i}(X(\real),\Z)$ and Stiefel-Whitney classes $w_{2i}\in \Hr^{2i}(X(\real),\Z/2)$. 

\begin{prop}\label{prop:commutcharacteristic}
If $X$ is a smooth real scheme, the following diagram commutes:
\[
\xymatrix{ \mathrm{GW}^2(X)\ar[rr]^-{b_i}\ar[dd]_-f\ar[rd] & &  \CHW^{2i}(X)\ar@{->}'[d][dd]\ar[rd]^-{\widetilde{\gamma_\real}} &  \\
& \mathrm{KU}^0(X(\real))\ar[rr]^(0.35){c_i^{\mathrm{top}}}\ar[dd] & & \Hr^{2i}(X(\real),\Z)\ar[dd] \\
\mathrm{K}_0(X)\ar@{->}'[r]^-{c_{2i}}[rr]\ar[rd] & & \CH^{2i}(X)\ar[rd]^-{\overline{\gamma_\real}} &   \\
 & \mathrm{KO}^0(X(\real))\ar[rr]_-{w_{2i}} & & \Hr^{2i}(X(\real),\Z/2)}
\]
\end{prop}

\begin{proof}
The front and back squares commute by construction of the respective characteristic classes. The left and right squares commute by definition of the respective realization maps, and it remains to prove the the top and bottom squares commute. We establish this for the top square; the arguments for the commutativity of the bottom square are similar. The main point of the proof is that if $\mathrm{H}\mathbb{P}^n$ is the quaternionic Grassmannian considered by Panin and Walter, then its realization is $\cplx\mathbb{P}^n$. This is true almost by definition, the former classifying rank $2$ subbundles of $\Osc^{2n+2}$ on which the restriction of the symplectic form $h_{2n+2}$ is non degenerate. This corresponds to a rank $1$ complex subbundle of $\cplx^{n+1}$. This fact generalizes to the quaternionic Grassmannian $\mathrm{H}\mathbb{P}(\mathscr{E})$ and we conclude easily. 
\end{proof}

\section{Corank $0$ vector bundles}\label{section:corank_zero}
Let $X$ be a smooth real algebraic variety of dimension $d$. Our goal in this section is to study the exact sequence of pointed sets \eqref{eqn:corank0}
\[
\mathcal{V}_{d-1}(X,\Lc)\xrightarrow{s_{d-1}}\mathcal{V}_{d}(X,\Lc)\xrightarrow{e}\CHW^n(X,\Lc)
\]
for any line bundle $\Lc$ over $X$. As discussed in Section \ref{sec:realrealization}, real realization induces a comparison map 
\[
r_n\colon \mathcal{V}_{n}(X,\Lc)\to \mathcal{V}^{\mathrm{top}}_n(X(\real),L)
\]
where the right-hand side denotes the pointed set of rank $d$ topological vector bundles on $X(\real)$ having determinant isomorphic to $\Lc(\real)=L$. 

\subsection{Statement}
Obstruction theory on the motivic side and on the topological side yield exact sequences of pointed sets that we can combine into a commutative diagram (use \cite[Proposition 6.1]{Hornbostel21} for the right-hand square)
\[
\xymatrix{\mathcal{V}_{d-1}(X,\Lc)\ar[r]^-{s_{d-1}}\ar[d]_-{r_{d-1}} & \mathcal{V}_{d}(X,\Lc)\ar[r]^-{e}\ar[d]_-{r_{d}}& \CHW^d(X,\Lc)\ar@{-->}[d]_-{\widetilde{\gamma_\real}(\Lc)} \\
\mathcal{V}^{\mathrm{top}}_{d-1}(X(\real),L)\ar[r]_-{s_{d-1}} & \mathcal{V}^{\mathrm{top}}_d(X(\real),L)\ar[r]_-{e(\real)} & \mathrm{H}^d(X(\real),\Z(L)).}
\]

\begin{thm}\label{thm:corank_zero}
Let $X$ be a smooth affine variety of dimension $d\geq 1$ over $\real$, and let $\mathscr{E}$ be a corank $0$ vector bundle on $X$. Then $\mathscr{E}$ splits off a trivial line bundle if, and only if, the top Chern class $c_d(\mathscr{E})$ vanishes and the topological vector bundle $\mathscr{E}(\real)$ splits off a trivial line bundle.
\end{thm}

\begin{rem}
As mentioned in the introduction, this result can be seen as a generalization of \cite[Theorem 4.30]{Bhatwadekar06}. Indeed, the above theorem applies for instance also to affine varieties $X$ of dimension $n$, which are isomorphic in $\mathcal{H}(\real)$ to an affine variety of dimension $d$.  As for the cases mentioned in \cite[Theorem 4.30]{Bhatwadekar06}, they follow from Proposition \ref{prop:chow--witt_as_fibre_product} and the following arguments:
\begin{enumerate}
\item If $X$ has no compact connected components, then $\Hr^d(X(\real),\Z(L))=0$ and the map $\CHW^d(X,\Lc)\rightarrow\CH^d(X)$ is an isomorphism.  In other words, vanishing of the top Chern class is a sufficient to guarantee splitting. 
\item If $d$ is even and $\omega_{X/\real}(\Rb)\vert_C\neq L\vert_C$ on every compact connected component $C$, then $\Hr^d(X(\real),\Z(L))$ is a $2$-torsion group so the reduction mod $2$ map from $\Hr^d(X(\Rb),\Zb(L))$ to $\Hr^d(X(\real),\Z/2)$ is an isomorphism. Therefore the map $\CHW^d(X,\Lc)\rightarrow\CH^d(X)$ is an isomorphism and again the top Chern class is a sufficient obstruction.
\item If $d$ is odd, then the topological Euler class $e(\Esc(\Rb))$ is $2$-torsion, and is therefore trivial on the compact connected components $C$ for which $\omega_{X/\real}(\Rb)\vert_C= L\vert_C$. On the other hand, it is equal to the top Stiefel-Whitney class $w_d(\Esc(\Rb))$ on the compact connected components $C$ where $\omega_{X/\real}(\Rb)\vert_C\neq L\vert_C$. Since $w_d(\Esc(\Rb))$ is the image of $c_d(\Esc)$ under the mod $2$ cycle class map $\CH^d(X)\rightarrow\Hr^d(X(\Rb),\Zb/2)$ by \cite[Théorème 4]{Kahn87}, the vanishing of $c_d(\Esc)$ guarantees that $w_d(\Esc(\Rb))=0$ and thus that $e(\Esc(\Rb))=0$. Therefore the vanishing of the top Chern class is once more a sufficient obstruction in case $d$ is odd.
\end{enumerate}
\end{rem}

\subsection{Proof of Theorem~\ref{thm:corank_zero}}
To prove the theorem, we first recall from Subsection \ref{subsection:real_cycle_class_maps} that there is a mod $2$ cycle class map 
\[
\overline{\gamma_\real}:\Hr^{d-1}(X,\KM_d)\rightarrow\Hr^{d-1}(X(\real),\Z/2)
\] 
defined using Colliot-Thélène and Parimala's signature mod $2$. We will need the following proposition.

\begin{prop}\label{prop:mod_2_cycle_surjective}
If the group $\CH^d(X_{\cplx})$ is $2$-torsion free, the homomorphism $\overline{\gamma_\real}$ is surjective. 
\end{prop}

\begin{proof}
If $X(\real)=\emptyset$, then the right-hand term is trivial and there is nothing to prove. We may thus suppose that $X(\real)\neq \emptyset$, in which case the result is \cite[Corollary 4.3.(b)]{Colliot96}.
\end{proof}

\begin{rem}
The assumption that $\CH^d(X_{\cplx})$ is $2$-torsion free is too strong, the precise condition being given by \cite[Theorem 4.2.(b)]{Colliot96}; it is, however, sufficient for our purposes. 
\end{rem}

\begin{rem}
In \cite[Section 4]{Colliot96}, Colliot-Thélène and Scheiderer use $\mathrm{K}^\mathrm{Q}$-cohomology where $\mathrm{K}^\mathrm{Q}$ is Quillen $K$-theory rather than Milnor $K$-theory; namely they consider $\Hr^{d-1}(X,\mathbf{K}_d^\mathrm{Q})$ where $\mathbf{K}_d^\mathrm{Q}$ is the sheaf associated with the presheaf $U\mapsto\mathrm{K}_d^\mathrm{Q}(U)$ instead of $\Hr^{d-1}(X,\KM_d)$. However, we have a natural isomorphism $\Hr^{d-1}(X,\KM_d)\cong\Hr^{d-1}(X,\mathbf{K}_d^\mathrm{Q})$: this follows from the fact that $\mathbf{K}_d^\mathrm{Q}$ and $\KM_d$ both admit a Gersten-type resolution and that there is a natural isomorphism $\mathrm{K}_i^\mathrm{M}(F)\cong\mathrm{K}_i^\mathrm{Q}(F)$ for any field $F$ if $i\leq 2$.
\end{rem}

\begin{prop}\label{prop:chow--witt_as_fibre_product}
Suppose that $\CH^d(X_{\cplx})$ is $2$-torsion free. Let $\Lc$ be a line bundle on $X$; denote by $L$ the topological line bundle $L=\Lc(\real)$. Then the morphism $\CHW^d(X,\Lc)\rightarrow\CH^d(X)\times\Hr^d(X(\real),\Z(L))$ provided by the comparison morphism and the real cycle class map induces an isomorphism 
\[
\CHW^d(X,\Lc)\cong\CH^d(X)\times_{\Hr^d(X(\real),\Z/2)}\Hr^d(X(\real),\Z(L))
\] 
of abelian groups.
\end{prop}

\begin{proof}
Consider the following commutative diagram with exact rows:
\[
\xymatrix{0 \ar[r] & \Ibf^{d+1}(\Lc) \ar[r] \ar[d]_-{\mathrm{sign}(\Lc)} & \KMW_d(\Lc) \ar[r] \ar[d]_-{\widetilde{\mathrm{sign}}(\Lc)} & \KM_d \ar[r] \ar[d]_-{\overline{\mathrm{sign}}} & 0 \\
0 \ar[r] & \iota_*\Z(L) \ar[r]_-{\cdot 2} & \iota_*\Z(L) \ar[r] & \iota_*\Z/2 \ar[r] & 0}
\]
of sheaves of abelian groups on $X$. It induces a commutative ladder with exact rows:
\[
\xymatrix{ \Hr^{d-1}(X,\KM_d) \ar[r] \ar[d]_-{\overline{\gamma_\real}} & \Hr^{d}(X,\Ibf^{d+1}(\Lc)) \ar[r] \ar[d] & \widetilde{\CH}^d(X,\Lc) \ar[r] \ar[d]  & \CH^d(X) \ar[d] \\
\Hr^{d-1}(X(\real),\Z/2) \ar[r]_-{\beta}   & \Hr^d(X(\real),\Z(L)) \ar[r]_-{\cdot 2}  & \Hr^d(X(\real),\Z(L)) \ar[r] & \Hr^d(X(\real),\Z/2).}
\]
In this diagram, the map $\overline{\gamma_\real}$ is surjective by Proposition \ref{prop:mod_2_cycle_surjective} and, by Jacobson's theorem \ref{theo:jacobson}, the real cycle class map $\Hr^d(X,\Ibf^{d+1}(\Lc))\rightarrow\Hr^d(X(\real),\Z(L))$ is an isomorphism. An elementary diagram chase then shows that the map $\widetilde{\CH}^d(X,\Lc)\to\Hr^d(X(\Rb),\Zb(L))\times_{\Hr^d(X(\Rb),\Zb/2)}\CH^d(X)$ of the lemma is injective. It is also surjective: indeed, the map $\widetilde{\CH}^d(X,\Lc)\to\CH^d(X)$ is onto, since its cokernel is a subgroup of $\Hr^{d+1}(X,\Ibf^{d+1}(\Lc))$ which vanishes as $X$ has dimension $d$, which easily implies the claimed surjectivity.
\end{proof}

\begin{proof}[Proof of Theorem \ref{thm:corank_zero}]
Let $\Lc=\det(\mathscr{E})^\vee$. If $d=1$, the result is obvious, and thus we may suppose that $d\geq 2$. Since $X$ is affine, the group $\CH^d(X_{\cplx})$ is uniquely divisible, and consequently the above proposition yields an isomorphism
\[
\CHW^d(X,\Lc)\cong\CH^d(X)\times_{\Hr^d(X(\real),\Z/2)}\Hr^d(X(\real),\Z(L)).
\] 
As noted earlier, the Euler class $e(\mathscr{E})$ maps to $(c_d(\mathscr{E}),e(\mathscr{E}(\real)))$ (where $e(\mathscr{E}(\real))$ is the Euler class of the topological line bundle $\mathscr{E}(\real)$) under the map $\widetilde{\CH}^d(X,\Lc)\rightarrow\CH^d(X)\times\Hr^d(X(\real),\Z(L))$. 

Now assume that $e(\mathscr{E})=0$. Then $c_d(\mathscr{E})=0$ and $e(\mathscr{E}(\real))=0$ by the above claim. Since $\mathscr{E}(\real)$ is of corank $0$, the vanishing of its Euler class implies that it splits off a trivial line bundle. Conversely, if $c_d(\mathscr{E})=0$ and $\mathscr{E}(\real)$ splits off a trivial line bundle, then $e(\mathscr{E}(\real))=0$ and therefore $e(\mathscr{E})=0$, showing that $E$ splits off a trivial line bundle as well.
\end{proof}


\section{The Euler class of corank 1 vector bundles}

In this section, we study the case of corank one vector bundles on a smooth affine variety $X$ of dimension $d\geq 3$. Once again, we have a commutative diagram 
\[
\xymatrix{\mathcal{V}_{d-2}(X,\Lc)\ar[r]^-{s_{d-2}}\ar[d]_-{r_{d-2}} & \mathcal{V}_{d-1}(X,\Lc)\ar[r]^-{e}\ar[d]_-{r_{d-1}}& \CHW^{d-1}(X,\Lc)\ar@{-->}[d]_-{\widetilde{\gamma_\real}(L)} \\
\mathcal{V}^{\mathrm{top}}_{d-2}(X(\real),L)\ar[r]_-{s_{d-2}} & \mathcal{V}^{\mathrm{top}}_{d-1}(X(\real),L)\ar[r]_-{e(\real)} & \mathrm{H}^{d-1}(X(\real),L(\real))}
\]
and we are interested in understanding the right-hand vertical map. 

\subsection{General results}

\begin{prop}\label{prop:colliotsch}
Let $X$ be a smooth real affine variety of dimension $d>1$ and let $\Lc$ be a line bundle on $X$, and let $T$ denote the set of compact connected components of $X(\real)$. The real cycle class map $\gamma^{d-1}_d$ induces a split exact sequence 
\[
0\rightarrow\bigoplus_T\Z/2\rightarrow\Hr^{d-1}(X,\Ibf^d(\Lc))\xrightarrow{\gamma^{d-1}_d}\Hr^{d-1}(X(\real),\Z(L))\rightarrow 0
\] 
of abelian groups.
\end{prop}

\begin{proof}
Let $\overline{\mathbf{K}}$ denote the kernel of the morphism $\cup(-1):\mathscr{H}^d\rightarrow\mathscr{H}^{d+1}$. Recall from \cite[Lemmas 3.1 and 3.2]{Lerbet26} that $\overline{\mathbf{K}}$ fits into an exact sequence 
\[
0\rightarrow\overline{\mathbf{K}}\rightarrow\Ibf^d(\Lc)\xrightarrow{\otimes\langle\langle -1\rangle\rangle}\Ibf^{d+1}(\Lc)\rightarrow 0
\] 
(in fact, the map from the kernel of the homomorphism $\otimes\langle\langle -1\rangle\rangle\colon\Ibf^d(\Lc)\rightarrow\Ibf^{d+1}(\Lc)$ to $\overline{\mathbf{K}}$ induced by the morphism $\Ibf^d(\Lc)\rightarrow\overline{\Ibf}^d\cong\mathscr{H}^d$ is an isomorphism). Since $X$ is affine of dimension $d\geq 2$, the étale cohomology group $\Hr_\et^{2d-1}(X_\cplx,\Z/2)$ vanishes, hence by \cite[Proposition 4.6]{Lerbet26} we have a split exact sequence 
\[
0\rightarrow\Hr^{d-1}(X,\overline{\mathbf{K}})\rightarrow\Hr^{d-1}(X,\Ibf^d(\Lc))\rightarrow\Hr^{d-1}(X,\Ibf^{d+1}(\Lc))\rightarrow 0
\] 
where the last non-trivial map is identified with the real cycle class map by Jacobson's isomorphism $\Ibf^{d+1}(\Lc)\cong\iota_*\Z(L)$. Hence it suffices to show that $\Hr^{d-1}(X,\overline{\mathbf{K}})\simeq\bigoplus_T\Z/2$ which follows from \cite[Corollary 3.2 (c)]{Colliot96}.
\end{proof}

The following proposition can be seen as a conditional analogue of Theorem \ref{thm:corank_zero}.
\begin{prop}\label{prop:corank_one_primary_obstruction}
Let $X$ be a smooth real affine variety of dimension $d\geq 2$, let $\mathscr{E}$ be a vector bundle of corank $1$ on $X$ and let $\Lc$ denote the dual of the determinant of $\mathscr{E}$. Assume that the mod $2$ real cycle class map 
\[
\overline{\gamma_\real}\colon \Hr^{d-2}(X,\KM_{d-1})\rightarrow\Hr^{d-2}(X(\real),\Z/2),
\]  
is surjective and that the homomorphism $\Hr^{d-1}(X,\Ibf^d(\Lc))\rightarrow{\CHW}^{d-1}(X,\Lc)$ induced by the inclusion $\Ibf^{d}(\Lc)\rightarrow\KMW_{d-1}(\Lc)$ vanishes on $\bigoplus_T\Z/2$, where $T$ is the set of compact connected components of $X(\real)$. Then the Euler class of $\mathscr{E}$ vanishes if, and only if, its top Chern class $c_{d-1}(\mathscr{E})$ and its topological Euler class $e(\mathscr{E}(\real))$ vanish.
\end{prop}

\begin{proof}
The commutative diagram with exact rows
\[
\xymatrix{0 \ar[r] & \Ibf^{d}(\Lc) \ar[r] \ar[d]_-{\mathrm{sign}(\Lc)} & \KMW_{d-1}(\Lc) \ar[r] \ar[d]_-{\widetilde{\mathrm{sign}}(\Lc)} & \KM_{d-1} \ar[r] \ar[d]_-{\overline{\mathrm{sign}}} & 0 \\
0 \ar[r] & \iota_*\Z(L) \ar[r]_-{\cdot 2} & \iota_*\Z(L) \ar[r] & \iota_*\Z/2 \ar[r] & 0}
\]
induces a commutative ladder with exact rows
\[
\xymatrix{\Hr^{d-2}(X,\KM_{d-1}) \ar[r] \ar[d]_-{\overline{\gamma_\real}} & \Hr^{d-1}(X,\Ibf^{d}(\Lc)) \ar[r] \ar[d] & {\CHW}^{d-1}(X,\Lc) \ar[r] \ar[d]   & \CH^{d-1}(X) \ar[d] \\
\Hr^{d-2}(X(\real),\Z/2) \ar[r]_-{\beta}   & \Hr^{d-1}(X(\real),\Z(L)) \ar[r]_-{\cdot 2}   & \Hr^{d-1}(X(\real),\Z(L)) \ar[r] & \Hr^{d-1}(X(\real),\Z/2).}
\]
Recall that the map $\Hr^{d-1}(X,\Ibf^d(\Lc))\rightarrow\Hr^{d-1}(X(\real),\Z(L))$ is a split epimorphism. Denote by $B$ the image of the splitting provided by Proposition \ref{prop:colliotsch}, which is a direct summand of $\Hr^{d-1}(X,\Ibf^d(\Lc))$. Then the following diagram:
\[
\xymatrix{\Hr^{d-2}(X,\KM_{d-1}) \ar[r] \ar[d]_-{\overline{\gamma_\real}} & B \ar[r] \ar[d] & {\CHW}^{d-1}(X,\Lc) \ar[r] \ar[d]      & \CH^{d-1}(X) \ar[d] \\
\Hr^{d-2}(X(\real),\Z/2) \ar[r]_-{\beta}      & \Hr^{d-1}(X(\real),\Z(L)) \ar[r]_-{\cdot 2}             & \Hr^{d-1}(X(\real),\Z(L)) \ar[r] & \Hr^{d-1}(X(\real),\Z/2)}
\]
where the map $\Hr^{d-2}(X,\KM_{d-1})\rightarrow B$ is the composite 
\[
\Hr^{d-2}(X,\KM_{d-1})\rightarrow\Hr^{d-1}(X,\Ibf^d(\Lc))=B\oplus\bigoplus_T\Z/2\rightarrow B
\] 
and the map $B\rightarrow{\CHW}^{d-1}(X,\Lc)$ is the restriction of the map $\Hr^{d-1}(X,\Ibf^{d+1}(\Lc))\rightarrow{\CHW}^{d-1}(X,\Lc)$, is commutative. Moreover, its rows are \emph{exact}. Indeed, we need only consider the top row. Exactness at ${\CHW}^{d-1}(X,\Lc)$ follows from the fact that since the map $\Hr^{d-1}(X,\Ibf^d(\Lc))\rightarrow{\CHW}^{d-1}(X,\Lc)=\bigoplus_T\Z/2\oplus B$ vanishes on $\bigoplus_T\Z/2$, its image is equal to the image of its restriction to $B$, and exactness at $B$ is clear. Now the map $B\rightarrow\Hr^{d-1}(X(\real),\Z(\Lc))$ is an isomorphism by construction and $\overline{\gamma_\real}$ is surjective. By diagram chase, one sees that the map 
\[
\widetilde{\CH}^{d-1}(X,\Lc)\rightarrow\CH^{d-1}(X)\times\Hr^{d-1}(X(\real),\Z(\Lc))
\] 
is injective. We then conclude as in the proof of Theorem \ref{thm:corank_zero}.
\end{proof}

\begin{rem}
It should be noted that the condition on $\overline{\gamma_\real}$ is very difficult to verify in practice. The second condition is satisfied if $X(\real)$ has no compact connected components, for instance if $X(\real)$ is empty (in which case the surjectivity of $\overline{\gamma_\real}$ is obvious).
\end{rem}

\subsection{Threefolds}
As written above, Proposition \ref{prop:corank_one_primary_obstruction} depends on the homomorphism $\overline{\gamma_\real}$. In this section, we prove that its conclusion can fail as hypotheses are eliminated. More precisely,  the goal of this section is to establish the following theorem.
\begin{thm}\label{thm:realthreefolds}
There exists an affine open subscheme $U$ of $\mathbb{P}_\real^3$ such that $U(\real)=\mathbb{P}^3(\real)$ and an explicit vector bundle $\mathscr{E}$ of rank $2$ on $U$ having the following properties:
\begin{itemize}[noitemsep,topsep=1pt]
\item $\det(\mathscr{E})$ is trivial.
\item $c_2(\mathscr{E})=0$.
\item $e(\mathscr{E})\neq 0\in{\CHW}^2(U)$.
\item $\mathscr{E}(\real)\simeq \mathbb{P}^3(\real)\times \real^2$.
\end{itemize}
\end{thm}

The proof of this result will rest on some delicate computations of Chow-Witt groups of hypersurfaces in $\PP^3_{\real}$, that will be the subject of the next subsection.

\subsubsection{Chow--Witt groups of complements of hypersurfaces in $\PP^3_{\real}$.}
The following type of result is classical to real algebraic geometers but we provide a proof for the convenience of the reader.

\begin{prop}\label{prop:existencedegreed}
If $d\geq 4$ is an even integer, there exists a smooth hypersurface $X\subset\mathbb{P}_\real^3$ of degree $d$ such that $X(\real)$ is empty and $\CH^1(X)$ is generated (as a group) by the class of the hyperplane section.
\end{prop}

\begin{proof}
Recall first the complex Noether--Lefschetz theorem. Let $V$ denote the vector space of homogeneous polynomials of degree $d$ in four variables with complex coefficients, namely $V=\mathbb{C}[x,y,z,t]_d$. Then the Noether--Lefschetz theorem states that the set $Z$ of $F\in V$ such that the complex hypersurface $X=\{F=0\}\subseteq\PP^3_{\mathbb{C}}$ has Picard group \emph{not} generated by $\mathscr{O}_X(1)$ is contained in a countable union of proper closed subvarieties of $V$. Up to adding to this countable union the discriminant variety parametrizing singular hypersurfaces, we may assume that any element of $V\setminus Z$ defines a smooth hypersurface. 

Let $G=\Z/2$, and endow $V=\mathbb{C}[x,y,z,t]_d$ with the conjugation action. The intersection of $Z$ with $V^G=\mathbb{R}[x,y,z,t]_d$ is also a countable union of proper real algebraic sets. In particular, its complement $V^G\setminus Z$, as a countable intersection of dense open subsets of $V^G$ for the Euclidean topology, is itself dense for the Euclidean topology. The set $U$ of $F\in V^G$ such that $F(x)>0$ for all $x\in\real^4\setminus 0$ is open for the Euclidean topology on $V^G$ according to \cite[Proposition 3.5]{Benoist18}. Assuming from now on that $d$ is even, $U$ is additionally non-empty since $x_0^d+x_1^d+x_2^d+x_3^d$ lies in $U$. It follows that $U$ meets $V^G\setminus Z$. In other words, there exists $F\in\mathbb{R}[x,y,z,t]_d$ such that: the complex hypersurface $S_F'=\{F=0\}\subseteq\PP^3_{\mathbb{C}}$ is smooth over $\mathbb{C}$, hence the real hypersurface $S_F=\{F=0\}\subseteq\PP^3_{\mathbb{R}}$ is smooth over $\mathbb{R}$; one has $F(x)>0$ for all $x\in\real^4\setminus 0$ so that $S_F(\real)$ is empty; and $\CH^1(S_F')$ is generated by the hyperplane section, which implies that $\CH^1(S_F)$ is also generated by the hyperplane section, for instance because $\CH^1(S_F)$ injects into $\CH^1(S_F')$ (e.g. \cite[Lemma 2.3]{Karpenko90}) and its image contains the hyperplane section. Then $X=S_F$ satisfies the conclusion of the proposition.
\end{proof}

We now fix a smooth hypersurface $i:X\hookrightarrow\PP_\real^3$ of degree $d$ defined by a polynomial $f$ such that $X(\real)=\emptyset$ and $\CH^1(X)$ is generated by the hyperplane section $i^*h=\OO_X(1)$.

\begin{lem}\label{lem:chow_groups_hypersurface}
We have $\CH^1(X)\simeq\Z$ and $\mathrm{Ch}^2(X)=\Z/2$ is generated by the class of any closed point.
\end{lem}

\begin{proof}
Since $X$ is proper with empty real locus, the equality $\mathrm{Ch}^2(X)=\Z/2$ follows from \cite[Theorem 1.3 (b)]{Colliot96}. Since $\mathrm{Ch}^2(X)$ is generated by the classes of closed points, this establishes the last assertion. To prove the first, since $\CH^1(X)$ is generated by $i^*h$, it suffices to prove that $i^*h$ is torsion free. Since $X$ is a hypersurface of degree $d$, the projection formula yields $i_*i^*h=dh^2$ in $\CH^2(\mathbb{P}^3)=\Z h^2$: in particular, $i^*h$ is torsion free as required.
\end{proof}

We now assume that $d=2n$ with $n$ \emph{odd}. Consider the square $h^2\in\CH^2(\mathbb{P}_\real^3)$ of the hyperplane section.

\begin{lem}\label{lem:computation_steenrod_square}
We have $i^*h^2\neq 0$ in $\mathrm{Ch}^2(X)=\Z/2$.
\end{lem}

\begin{proof}
Note that $h^2$ is the class of any projective line $L=\PP^1_\real$ in $\PP_\real^3$. Since $X$ has empty real locus, and $L$ is not contained in $X$, we deduce that $X\cap L$ is a finite set of points. In particular, the codimension of $X\cap L$ in $X$ is equal to the codimension of $L$ in $\mathbb{P}^3_\real$, namely both are equal to $2$; moreover $L$ is regular and thus Cohen--Macaulay. Consequently, $i^*h^2$ is equal to the class of $X\cap L$ in $\mathrm{Ch}^2(X)$. In other words, it is the count mod $2$ of closed points in $X\cap L$. These points correspond to the irreducible factors of $f$ when seen as an element of $\real[t]$ using an isomorphism between $L$ and the projective line. Since $X$ has no real point, $f$ has $n$ factors of degree $2$, hence $X\cap L$ is composed of $n$ closed points. As we assumed that $n$ is odd, $[X\cap L]=i^*h^2\in\mathrm{Ch}^2(X)$ is non-zero.
\end{proof}

Now denote by $U$ the open complement of $X$ and consider the following diagram:
\[
\xymatrix{{\CHW}^1(X) \ar[r] \ar[d] & {\CHW}^2(\mathbb{P}_\real^3) \ar[r] \ar[d] & {\CHW}^2(U) \ar[r] \ar[d] & \Hr^2(X,\KMW_1) \\
\CH^1(X) \ar[r]\ar[d] & \CH^2(\mathbb{P}_\real^3) \ar[r]                       & \CH^2(U) & \\
\Hr^2(U,\Ibf^2) \ar[d]                &                                                        &  & & \\
\Hr^2(U,\KMW_1)}
\]
with exact lines and columns, and whose lines are localisation exact sequences (since $X$ is of even degree and $\mathbb{P}_\real^3$ is of odd dimension, the canonical sheaves of both are squares and hence there is no twist in the cohomology groups of $X$). 

\begin{lem}\label{lem:Steenrod}
The map $\widetilde{\CH}^1(X)\rightarrow\CH^1(X)=\Z i^*h$ has image $2\CH^1(X)$.
\end{lem}

\begin{proof}
We have a commutative diagram:
\[
\xymatrix{{\CHW}^1(X) \ar[r] & \CH^1(X) \ar[r]^-{\partial}\ar[d]       & \Hr^2(X,\Ibf^2) \ar[r] \ar@{=}[d] & \Hr^2(X,\KMW_1) \\
                                & \Hr^1(X,\overline{\Ibf}) \ar[r] \ar[rd]_-{\mathrm{Sq}^2} & \Hr^2(X,\Ibf^2) \ar[d] & \\
                                 &                                                 & \Hr^2(X,\overline{\Ibf}^2) &}
\]
whose diagonal map is the first Steenrod square $\mathrm{Sq}^2$ of Brosnan (\cite[Theorem 1.1]{Totaro03}) and thus coincides with the self-intersection, or squaring, operation on $\Hr^1(X,\overline{\Ibf})\simeq\Ch^1(X)$. The map $\Hr^2(X,\Ibf^2)\rightarrow\Hr^2(X,\overline{\Ibf}^2)$ is an isomorphism. Indeed, since $X$ has no real points, the sheaf $\Ibf^3$ itself vanishes on $X$ (for instance because $\Ibf^3(Y)=\Hr^0(Y(\real),\Z)=0$ for all $Y\subset X$ open by Jacobson's theorem). Modulo this isomorphism, we have $\partial (x)=\mathrm{Sq}(\overline{x})$ for any $x\in\CH^1(X)$ with reduction mod $2$ given by $\overline{x}$. Since $i^*h$ lies in codimension $1$, $\mathrm{Sq} i^*\overline{h}=i^*\overline{h}^2$ and is thus non-zero by Lemma \ref{lem:computation_steenrod_square}. In particular, $h$ does not lie in the image of the comparison map ${\CHW}^1(X)\rightarrow\CH^1(X)$. Since the cokernel of this comparison map is a subgroup of $\Hr^2(X,\Ibf^2)=\Hr^2(X,\overline{\Ibf}^2)$, which is a $2$-torsion group, the image of the comparison map contains $2\CH^1(X)$ and is consequently equal to $2\CH^1(X)$.
\end{proof}

The proof also supplies us with the following vanishing statement.

\begin{cor}\label{cor:vanishsurfaces}
The group $\Hr^2(X,\KMW_1)$ vanishes.
\end{cor}

\begin{proof}
Indeed, $\partial i^*h=\mathrm{Sq} i^*h\neq 0$ modulo the identification explained in the previous proof, hence the homomorphism $\partial:\CH^1(X)\rightarrow\Hr^2(X,\Ibf^2)=\Z/2$ is surjective. Now we have an exact sequence 
\[
\CH^1(X)\xrightarrow{\partial}\Hr^2(X,\Ibf^2)\rightarrow\Hr^2(X,\KMW_1)\rightarrow\Hr^2(X,\KM_1).
\] 
Examination of the Rost complex of $\KM_1$ shows that $\Hr^2(X,\KM_1)=0$ and the surjectivity of $\partial$ implies that $\Hr^2(X,\KMW_1)=0$ as required. 
\end{proof}

We recall that the comparison morphism ${\CHW}^2(\mathbb{P}_\real^3)\rightarrow\CH^2(\mathbb{P}_\real^3)$ is an isomorphism \cite[Corollary 11.8]{Fasel09d}.

\begin{prop}
The pullback morphism ${\CHW}^2(\mathbb{P}^3_\real)\rightarrow{\CHW}^2(U)$ (respectively $\CH^2(\mathbb{P}^3_\real)\rightarrow\CH^2(U)$) induces an isomorphism ${\CHW}^2(U)\cong\Z/2d\Z$ (respectively $\CH^2(U)\cong\Z/d\Z$). 
\end{prop}
\begin{proof}
Consider the following commutative square:
\[
\xymatrix{{\CHW}^1(X) \ar[r] \ar[d] &{\CHW}^2(\mathbb{P}_\real^3)=\Z \ar@{=}[d] \\
\CH^1(X)=\Z \ar[r]          & \CH^2(\mathbb{P}_\real^3)=\Z.}
\]
The bottom horizontal morphism is multiplication by $d$: indeed, the source is generated by $i^*h$ which is sent to $dh\in\CH^2(\mathbb{P}_\real^3)$ by $i_*$ as noted in the proof of Lemma \ref{lem:chow_groups_hypersurface}. In particular, the localisation exact sequence for Chow groups yields $\CH^2(U)=\Z/d\Z$, generated by the pull-back of $h^2$. 

On the other hand, the left vertical map has image $2\Z$, thus the composite 
\[
{\CHW}^1(X)\rightarrow\CH^1(X)\rightarrow\CH^2(\mathbb{P}_\real^3)
\] 
has image $2d\Z$. Since the right vertical map is an isomorphism, we conclude that the map ${\CHW}^1(X)\rightarrow{\CHW}^2(\mathbb{P}_\real^3)$ has image $2d\Z$. The localisation exact sequence 
\[
{\CHW}^1(X)\rightarrow{\CHW}^2(\mathbb{P}_\real^3)\rightarrow{\CHW}^2(U)\rightarrow\Hr^2(X,\KMW_1)=0
\] 
then shows that the map ${\CHW}^2(\mathbb{P}_\real^3)\rightarrow{\CHW}^2(U)$ induces an isomorphism $\Z/2d\rightarrow{\CHW}^2(U)$. 
\end{proof}

We now have the following commutative diagram:
\[
\xymatrix{{\CHW}^2(U)=\Z/2d \ar[d] \ar[r] & \CH^2(U)=\Z/d \ar[d] \\
\Hr^2(U(\real),\Z)      \ar[r]              & \Hr^2(U(\real),\Z/2)}
\]
where $U(\real)=\mathbb{P}^3(\real)$ as $X(\real)=\emptyset$ so that $\Hr^2(U(\real),\Z)=\Z/2$. The class $\overline{d}\in\Z/2d$ obviously maps to $0$ under the top horizontal morphism; since $d$ is even, it also maps to $0$ under the left-hand vertical map.

To establish the first three assertions of Theorem \ref{thm:realthreefolds}, it suffices to construct a rank $2$-vector bundle on $U$ with trivial determinant and Euler class $\overline{d}\in\Z/2d$. This will be done in the following subsection. Before doing so, let us conclude this section with the following lemma which identifies the hypotheses of Proposition \ref{prop:corank_one_primary_obstruction} that fail to be satisfied in this example.

\begin{lem}
The map 
\[
\overline{\gamma_\real}:\Hr^{1}(U,\KM_2)\rightarrow\Hr^1(U(\real),\Z/2)
\] 
is surjective.
\end{lem} 

\begin{proof}
The real cycle class map induces a commutative square:
\[
\xymatrix{\Hr^1(\mathbb{P}^3,\KM_2) \ar[r] \ar[d] & \Hr^1(U,\KM_2) \ar[d] \\
\Hr^1(\mathbb{P}^3(\real),\Z/2) \ar[r]          & \Hr^1(U(\real),\Z/2)}
\]
whose bottom horizontal arrow is an isomorphism. Thus it is sufficient to show that the homomorphism $\Hr^1(\mathbb{P}^3,\KM_2)\rightarrow\Hr^1(\mathbb{P}^3(\real),\Z/2)$ is surjective. Multiplication by the generator $\{-1\}$ of Milnor $\mathrm{K}$-theory induces a morphism $\KM_1\rightarrow\KM_2$ and a commutative triangle:
\[
\xymatrix{\Hr^1(\mathbb{P}^3,\KM_1) \ar[r] \ar[rd] & \Hr^1(\mathbb{P}^3,\KM_2) \ar[d] \\
                                                    & \Hr^1(\mathbb{P}^3(\real),\Z/2)}
\]
Thus it suffices to show that the map $\Hr^1(\mathbb{P}^3,\KM_1)=\CH^1(X)\rightarrow\Hr^1(\mathbb{P}^3(\real),\Z/2)$ is surjective. In fact, it sends $c_1(\mathscr{O}_{\mathbb{P}^3}(1))$ to $w_1(\mathscr{O}_{\mathbb{P}^3}(1)(\real))$ by \cite[Théorème 4]{Kahn87}. The line bundle $\mathscr{O}_{\mathbb{P}^3}(1)(\real)$ is the tautological line bundle on $\mathbb{P}^3(\real)$ and is thus non-trivial so that $w_1(\mathscr{O}_{\mathbb{P}^3}(1)(\real))\neq 0$. This concludes the proof since $\Hr^1(\mathbb{P}^3(\real),\Z/2)=\Z/2$ is generated by $w_1(\mathscr{O}_{\mathbb{P}^3}(1)(\real))$.
\end{proof}

Thus, Proposition \ref{prop:corank_one_primary_obstruction} cannot be applied precisely because the restriction of the homomorphism $\Hr^2(U,\Ibf^3)\rightarrow\widetilde{\CH}^2(U)$ to the $\Z/2$ factor in Proposition \ref{prop:colliotsch} that does not come from the real cycle class map is not zero. On the contrary, it is injective and in fact, since the image is $2$-torsion, this restriction maps the non-zero element of $\Z/2$ to $\overline{d}\in\Z/2d\Z$.

\subsubsection{Construction of vector bundles on $\mathbb{P}^3$.}
In this section, we work with the affine variety $U\subset \PP^3_{\real}$ which is the complement of the hypersurface $X\subset \PP^3$ of degree $d=2n$ constructed in the previous section. For any $m\in\ZZ/2d\ZZ$, we exhibit a rank $2$ bundle $V(m)$ on $U$ whose Euler class is $m$, concluding the proof of the three first assertions of Theorem \ref{thm:realthreefolds}.  In the sequel, we denote by $x,y,z,t$ the coordinates on $\PP^3_{\real}$. Consider the following diagram
\[
\xymatrix@C=3em{0\ar[r] & \OO(-1)\ar[r]^-{\tiny{\begin{pmatrix} x \\ y \\ z \\ t\end{pmatrix}}}\ar@{-->}[d]_-f & \OO^{4}\ar[r]\ar[d] & \mathcal G\ar[r]\ar@{-->}[d]^-{-f^\vee} & 0 \\
0\ar[r] & \mathcal G^\vee\ar[r] & (\OO^{4})^\vee\ar[r]_-{\tiny{\begin{pmatrix} x & y & z & t\end{pmatrix}}} & \OO(1)\ar[r] & 0
}
\]
in which the second line is the dual of the first, and the middle vertical morphism is given by the standard symplectic form, given by the matrix
\[
{\begin{pmatrix} 0 & 1 & 0 & 0 \\ -1 & 0 & 0 & 0 \\ 0 & 0 & 0 & 1\\ 0 & 0 & -1 & 0\end{pmatrix}}.
\]
We obtain induced morphisms indicated by the dotted arrows, which are respectively injective and surjective (by the snake lemma).  We denote by $\mathcal{F}$ the kernel of $-f^\vee$, and we observe that the snake lemma yields a skew-symmetric isomorphism $\varphi:\mathcal{F}\to \mathcal{F}^\vee$. In particular, $\mathcal{F}$ has trivial determinant (the trivialization being given by the symplectic form).

\begin{lem}\label{lem:EulerclassG}
The Euler class $e(\mathcal{F})$ is a generator of $\widetilde{\mathrm{CH}}^2(\PP_\Rb^3)$.
\end{lem}

\begin{proof}
The projection $\widetilde{\mathrm{CH}}^2(\PP_\Rb^3)\to \mathrm{CH}^2(\PP_\Rb^3)$ is an isomorphism (independently of the base field) by \cite[Corollary 11.8]{Fasel09d}. Under this map, the Euler class is sent to the second Chern class $c_2(\mathcal{F})$ and it suffices to compute this class. Using the Whitney formula, the exact sequence 
\[
0\to \OO(-1)\to \OO^{4}\to \mathcal G\to 0
\]
yields $c_1(\mathcal G)=c_1(\OO(1))=-h$, while $c_2(\mathcal G)=-c_1(\OO(1))c_1(\OO(-1))=h^2$. Using now the exact sequence
\[
0\to \mathcal F\to \mathcal G\to \OO(1)\to 0
\]
and the fact that $c_1(\mathcal F)=0$, we obtain $c_2(\mathcal G)=c_2(\mathcal F)=h^2$. The claim follows.
\end{proof}

Recall next that for a symplectic bundle of rank $2$, the Euler class coincides with its first Borel class $b_1$, as defined by Panin-Walter. Moreover, the first Borel class
\[
b_1\colon\mathrm{GW}^2(\PP^3)\to \widetilde{\mathrm{CH}}^2(\PP^3)
\]
is a group homomorphism. It follows that $b_1(m\cdot [\mathcal{G},\varphi])=m\in \widetilde{\mathrm{CH}}^2(\PP^3)$. 

\begin{rem}
We could also use Horrocks's construction \cite{Horrocks67}, to obtain a rank $2$ bundle $E$ on $\mathbb{P}^3$ with $c_1(E)=0$ and $c_2(E)=mh^2$ for any $m\in \Z$.
\end{rem}

Now, we use the commutative diagram
\[
\xymatrix{\mathrm{GW}^2(\PP^3)\ar[r]^-{b_1}\ar[d] & \widetilde{\mathrm{CH}}^2(\PP^3)\ar[d] \\
\mathrm{GW}^2(U)\ar[r]_-{b_1} & \widetilde{\mathrm{CH}}^2(U)}
\]
and the fact that the right-hand vertical map is onto to obtain that the Borel class of the restriction $[\mathcal{G}_U,\varphi]$ of $[\mathcal{G},\varphi]$ to $U$ is a generator of $\widetilde{\mathrm{CH}}^2(U)$. To conclude, it suffices to prove that for any $m\in \N$ the class of $m[\mathcal{G}_U,\varphi]$ in $\mathrm{GW}^2(U)$ is represented by a rank $2$ symplectic bundle. This is a classical statement (e.g. \cite[Proposition 11]{Fasel09c}), using the fact that $U$ is affine.


\section{Secondary obstructions}\label{section:secondary_obstructions}

We now deal with the (motivic) secondary obstruction explained in Section \ref{sec:motivicsplitting}. Recall that the relevant homotopy sheaf is $\piaone_{d-1}(\A^{d-1}\smallsetminus \{0\})$ which sits in an exact sequence of strictly $\A^1$-invariant sheaves involving the degree map $\piaone_{d-1}(\A^{d-1}\smallsetminus\{0\})\to \mathbf{GW}_{d}^{d-1}$ for $d\geq 3$. If $X$ is a smooth affine $d$-fold over $\real$, this map induces for any line bundle $\Lc$ over $X$ a homomorphism
\[
\Hr^d(X,\piaone_{d-1}(\A^{d-1}\smallsetminus\{0\})(\Lc)) \longrightarrow \Hr^d(X,\mathbf{GW}^{d-1}_d(\Lc))
\]
which is onto by \cite[Theorem 4.4.5]{Asok14b}. We then refer to the right-hand side as the \emph{first part of the (secondary) obstruction}. A computation of the group $\Hr^d(X,\mathbf{GW}^{d-1}_d(\Lc))$ was obtained in \cite[Theorem 3.7.1]{Asok12c}: it participates in an exact sequence of the form
\begin{equation}\label{eqn:firstpartobstruction}
\Ch^{d-1}(X)\xrightarrow{\mathrm{Sq}^2_{\Lc}}\Ch^d(X)\longrightarrow \Hr^{d}(X,\mathbf{GW}^{d-1}_d(\Lc))\rightarrow 0,
\end{equation}
in which $\mathrm{Sq}^2_\Lc$ is the Steenrod square twisted by $\Lc$, that is, 
\[
\mathrm{Sq}^2_\Lc(-)=\mathrm{Sq}^2(-)+\overline{c_1}(\Lc)\cup(-)
\] 
where $\mathrm{Sq}^2$ is the Steenrod square defined by Voevodsky in \cite{Voevodsky03} and $\overline{c_1}(\Lc)$ is the reduction mod $2$ of the first Chern class of $\Lc$. The object of the next section is to compute the group $\Hr^{d}(X,\mathbf{GW}^{d-1}_d(\Lc))$ in terms of topological data.
\subsection{First part of the obstruction}

\subsubsection{A comparison with topology.}
Recall that for any topological space $M$ (say having the homotopy type of a CW complex), the connecting homomorphism in the long exact sequence in cohomology associated with the exact sequence
\[
0 \rightarrow\Z\xrightarrow{2}\Z\rightarrow\Z/2\rightarrow 0,
\]
is called the integral Bockstein homomorphism $\beta:\Hr^*(M,\Z/2)\rightarrow\Hr^{*+1}(M,\Z)$.  When composed with the reduction mod $2$ homomorphism, this yields a stable operation $\mathrm{Sq}^1:\Hr^*(M,\Z/2)\rightarrow\Hr^{*+1}(M,\Z/2)$ \cite[Chapter I]{steenrod-epstein-1962}, which is also frequently written $\beta$.

More generally, given a real line bundle $L$ on $M$, classified by a map $\xi: M \to \mathbb{RP}^{\infty}$ (up to homotopy), there is an induced orientation character $\pi_1(M) \to \Z/2$ and we may form cohomology with local coefficients twisted by this orientation character.  In this case, the exact sequence in cohomology with local coefficients attached to a twisted version of the exact sequence above yields a twisted (integral) Bockstein map $\beta_L:\Hr^*(M,\Z/2)\rightarrow\Hr^*(M,\Z(L))$.  As before, there is a corresponding reduction modulo $2$ map, and the composite of $\beta_L$ with the reduction modulo $2$ map, denoted $\mathrm{Sq}^1_L$, can be expressed in terms of the untwisted Bockstein and the first Stiefel--Whitney class of $L$ via the formula:
\[
\mathrm{Sq}^1_L = \mathrm{Sq}^1(-)+w_1(L)\cup(-);
\]
see, e.g., \cite[Theorem 2.3]{Greenblatt06} for this compatibility.  The next result relates the real realization of the twisted version of Voevodsky's Steenrod squaring operation and the twisted Bockstein we just described.



\begin{lem}
Let $X$ be a smooth real algebraic variety, let $\Lc$ be a line bundle on $X$ and let $n>0$ be an integer. Then the following square:
\[
\xymatrix{\Ch^{n-1}(X) \ar[r]^-{\mathrm{Sq}^2_\Lc} \ar[d]_-{\overline{\gamma_\real}} & \Ch^n(X) \ar[d]^-{\overline{\gamma_\real}} \\
\Hr^{n-1}(X(\real),\Z/2) \ar[r]_-{\mathrm{Sq}^1_L}       & \Hr^n(X(\real),\Z/2)}
\]
commutes, where $L=\Lc(\real)$.
\end{lem}

\begin{proof}
Since the real cycle class map is compatible with cup-products, via the compatibilities described above for the twisted squaring operations, it suffices to establish that $\overline{\gamma_\real}\circ\mathrm{Sq}^2=\mathrm{Sq}^1\circ\overline{\gamma_\real}$ and $\overline{\gamma_\real}\overline{c_1}(\Lc)=w_1(L)$, where $\overline{c_1}(\Lc)$ is the reduction modulo two of the first Chern class. The second statement can be found in \cite[Théorème 4]{Kahn87}. Let us verify the first statement. Recall \cite[Corollary 4.1.3]{Asok13c} that Voevodsky's Steenrod square can be described via the following composition:
\[
\xymatrix{\Ch^{n-1}(X)=\Hr^{n-1}(X,\overline{\Ibf}^{n-1}) \ar[r]^-{\partial} & \Hr^n(X,\Ibf^n) \ar[d] \\
                                                                    & \Hr^n(X,\overline{\Ibf}^n)=\Ch^n(X)}
\]
where $\partial$ is the connecting homomorphism induced by the exact sequence 
\[0
\rightarrow\Ibf^n\rightarrow\Ibf^{n-1}\rightarrow\overline{\Ibf}^{n-1}\rightarrow 0
\] 
and the vertical map is induced by the quotient morphism $\Ibf^n\rightarrow\overline{\Ibf}^n$. Consider now the following diagram:
\[
\xymatrix{\Hr^{n-1}(X,\overline{\Ibf}^{n-1}) \ar[r]^-{\partial} \ar[d]_-{\overline{\gamma_\real}} & \Hr^n(X,\Ibf^n) \ar[r] \ar[d]^-{\gamma_\real} & \Hr^n(X,\overline{\Ibf}^n) \ar[d]^-{\overline{\gamma_\real}} \\
\Hr^{n-1}(X(\real),\Z/2)  \ar[r]_-\beta     & \Hr^n(X(\real),\Z) \ar[r]        & \Hr^n(X(\real),\Z/2)  }
\]
where the composition of the maps in the first (respectively second) row is the algebraic (respectively topological) Steenrod square. The two inner squares commute because they are obtained using the morphisms of exact sequences of Diagram (\ref{diag:morphism_exact_sequence_real}) for the integers $n-1$ and $n$. Thus the outer square commutes as required.
\end{proof}

\begin{cor}\label{cor:comparison_secondary_obstruction_group_first_part}
Assume that $X$ is not proper or that $X(\real)$ is not empty. Denote by $d$ the dimension of $X$; let $\Lc$ be a line bundle on $X$ with real realization $L=\Lc(\real)$. Then the real cycle class map $\overline{\gamma_\real}:\Ch^d(X)\rightarrow\Hr^d(X(\real),\Z/2)$ induces an isomorphism 
\[
\Hr^d(X,\mathbf{GW}_d^{d-1}(\Lc))\rightarrow\Hr^d(X(\real),\Z/2)/\mathrm{Sq}^1_L\overline{\gamma_\real}(\Ch^{d-1}(X))
\] 
on cokernels and thus a surjective morphism $\Hr^d(X,\mathbf{GW}_d^{d-1}(\Lc))\rightarrow\mathrm{Coker}(\mathrm{Sq}^1_L)$.
\end{cor}

\begin{proof}
The previous lemma provides a commutative diagram with exact rows:
\[
\xymatrix{0 \ar[r] & \mathrm{Im}(\mathrm{Sq}^2_\Lc) \ar[r] \ar[d] & \Ch^d(X) \ar[r] \ar[d]^-{\overline{\gamma_\real}}& \Hr^d(X,\mathbf{GW}_d^{d-1}(\Lc)) \ar[r] \ar@{-->}[d] & 0 \\
0 \ar[r] & \mathrm{Sq}^1_L\overline{\gamma_\real}(\Ch^{d-1}(X)) \ar[r]          & \Hr^d(X(\real),\Z/2) \ar[r]     & \Hr^d(X(\real),\Z/2)/\mathrm{Sq}^1_L\overline{\gamma_\real}(\Ch^{d-1}(X)) \ar[r] & 0}
\]
This shows that $\overline{\gamma_\real}:\Ch^d(X)\rightarrow\Hr^d(X(\real),\Z/2)$ does induce a map as in the lemma---this is the dotted arrow in the diagram. The map $\overline{\gamma_\real}$ is an isomorphism by \cite[Theorem 1.3 (b)]{Colliot96}; the snake lemma shows that the map on cokernels is surjective and that its kernel is isomorphic to the cokernel of the map $\overline{\gamma_\real}:\mathrm{Im}(\mathrm{Sq}^2_\Lc)\rightarrow\mathrm{Sq}^1_L\overline{\gamma_\real}(\Ch^{d-1}(X))$ which is zero since $\overline{\gamma_\real}\circ\mathrm{Sq}^2_\Lc=\mathrm{Sq}^1_L\circ\overline{\gamma_\real}$.
\end{proof}

We shall see in Subsection \ref{subsubsection:explicit_example_non-injective} an explicit example showing that the surjection $\Hr^d(X,\GW_d^{d-1}(\Lc))\rightarrow\mathrm{Coker}(\mathrm{Sq}_L)$ is not injective in general. Let us however explain a positive result first.

\begin{prop}
Let $M$ be a smooth manifold of dimension $d$ with orientation bundle $\omega_M$ and let $L$ be a line bundle on $M$. If $L_{|C}\simeq(\omega_M)_{|C}$ for any compact connected component $C$ of $M$, then 
\[
\mathrm{Sq}^1_L:\Hr^{d-1}(M,\Z/2)\rightarrow\Hr^d(M,\Z/2)
\] 
vanishes.
\end{prop}

\begin{proof}
Indeed, $\mathrm{Sq}^1_L$ factors through the group $\Hr^d(M,\Z(L))$ as previously observed. The group $\Hr^d(M,\Z(L))$ is isomorphic to the free abelian group on compact connected components of $M$ by twisted Poincaré duality and is thus torsion free. Consequently, $\mathrm{Sq}^1_L$ vanishes as required. 
\end{proof}

\begin{cor}
Let $X$ be a smooth real algebraic variety and let $\Lc$ be a line bundle on $X$ with real realization $\Lc(\real)=L$. Assume that $L_{|C}\simeq\omega_{X/\real}(\real)_{|C}$ for any compact connected component $C$ of $X(\real)$. Then the map $\Hr^d(X,\GW_d^{d-1}(\Lc))\rightarrow\mathrm{Coker}(\mathrm{Sq}^1_L)$ is an isomorphism.
\end{cor}

\begin{proof}
This follows for instance from the morphism of exact sequences of the proof of Corollary \ref{cor:comparison_secondary_obstruction_group_first_part} since the groups $\mathrm{Im}(\mathrm{Sq}^1_L)$ and $\mathrm{Sq}^1_L(\overline{\gamma_\real}(\Ch^{d-1}(X)))$ both vanish in this case.
\end{proof}

\subsubsection{An example.}\label{subsubsection:explicit_example_non-injective}

We now exhibit a situation where the homomorphism $\Hr^d(X,\GW_d^{d-1}(\Lc))\rightarrow\mathrm{Coker}(\mathrm{Sq}^1_L)$ is not injective. In this subsection, given a smooth manifold $N$, we write $\Hr^i(N)$ for $\Hr^i(N,\Z/2)$ (singular cohomology with coefficients in $\Z/2$).

Let $M$ denote the disjoint union of two copies of the circle $S^1$ and let $\mathbb{T}=S^1\times S^1$ denote the $2$-dimensional torus. The cohomology of $M$ with coefficients in $\Z/2$ may be written as 
\[
\Hr^0(M)=\Z/2\cdot s\oplus\Z/2\cdot t,\;\Hr^1(M)=\Z/2\cdot  x\oplus\Z/2\cdot y
\] 
where $s\cup x=x$, $s\cup y=0$, $t\cup x=0$ and $t\cup y=y$. The Künneth formula gives an isomorphism 
\[
\Hr^l(M\times\mathbb{T})=\bigoplus_{i+j=l}\Hr^i(M)\otimes\Hr^j(\mathbb{T})
\] 
induced by sending $\alpha\otimes\beta\in\Hr^i(M)\otimes\Hr^j(\mathbb{T})$ to $p^*\alpha\cup q^*\beta$, where $p:M\times\mathbb{T}\rightarrow M$ and $q:M\times\mathbb{T}\rightarrow\mathbb{T}$ are the projections. In particular, this yields an element $w=p^*(x+y)\cup q^*\beta_0$ where $\beta_0\in\Hr^0(\mathbb{T})$ is the unit: it is the first Stiefel--Whitney class of the topological line bundle $L$ on $M\times\mathbb{T}$ obtained by pulling back the line bundle on $M$ which is non-trivial on both connected components of $M$. 

We also have by the Künneth formula a decomposition 
\[
\Hr^2(M\times\mathbb{T})=\Hr^1(M)\otimes\Hr^1(\mathbb{T})\oplus\Hr^0(M)\otimes\Hr^2(\mathbb{T})
\] 
since $M$, being a smooth manifold of dimension $1$, has no cohomology in degree $2$. We denote by $\beta_2$ the generator of $\Hr^2(\mathbb{T})$.  The following fact summarizes a straightforward computation.

\begin{lem}
The Bockstein $\mathrm{Sq}^1:\Hr^2(M\times\mathbb{T})\rightarrow\Hr^3(M\times\mathbb{T})$ vanishes. Moreover, cup-product with $w$ vanishes on $\Hr^1(M)\otimes\Hr^1(\mathbb{T})$ and it induces an isomorphism $\Hr^0(M)\otimes\Hr^2(\mathbb{T})\rightarrow\Hr^3(M\times\mathbb{T})$ sending $p^*s\cup q^*\beta_2$ to $p^*x\cup q^*\beta_2$ and $p^*t\cup q^*\beta_2$ to $p^*y\cup q^*\beta_2$.
\end{lem}


Let $C$ be the projective real elliptic curve given by the affine equation $y^2=x^3-x$, so that $C(\real)=M$. Recall that \cite[Theorem 1.3 (b)]{Colliot96} yields an isomorphism $\Ch^1(C)=\Hr^1(M)$. In particular, there exists a line bundle on $C$ which is non-trivial on both connected components of $M$; its pullback $\Lc$ to $C\times\mathbb{P}^1\times\mathbb{P}^1$ realizes to $L=\Lc(\real)$. In particular, $c_1(\Lc)$ maps to $w_1(L)$ under the Borel--Haefliger class map $\CH^1(C\times\mathbb{P}^1\times\mathbb{P}^1)\rightarrow\Hr^1(C(\real)\times\mathbb{P}^1(\real)\times\mathbb{P}^1(\real))$ by \cite[Th\'eor\`eme 4]{Kahn87}.

\begin{lem}
The image of the Borel--Haefliger cycle class map 
\[
\Ch^2(C\times\mathbb{P}^1\times\mathbb{P}^1)\rightarrow\Hr^2(C(\real)\times\mathbb{P}^1(\real)\times\mathbb{P}^1(\real))=\Hr^1(M,\Z/2)\otimes\Hr^1(\mathbb{T})\oplus\Hr^0(M)\otimes\Hr^2(\mathbb{T})
\] 
is $\Hr^1(M)\otimes\Hr^1(\mathbb{T})\oplus\Z/2\cdot (s+t)\otimes\Hr^2(\mathbb{T})$. In particular, 
\[
\mathrm{Sq}^1_L\overline{\gamma_\real}(\Ch^2(C\times\mathbb{P}^1\times\mathbb{P}^1))\neq\Hr^3(C(\real)\times\mathbb{P}^1(\real)\times\mathbb{P}^1(\real)).
\]
\end{lem}

\begin{proof}
Since $\mathbb{P}^1\times\mathbb{P}^1$ has an algebraic cell decomposition, we have an isomorphism 
\[
\Ch^l(C\times\mathbb{P}^1\times\mathbb{P}^1)\cong\bigoplus_{i+j=l}\Ch^i(C)\oplus\Ch^j(\mathbb{P}^1\times\mathbb{P}^1)
\] 
given by $\alpha\otimes\beta\mapsto p^*\alpha\cup q^*\beta$ with $p:C\times\mathbb{P}^1\times\mathbb{P}^1\rightarrow C$ and $q:C\times\mathbb{P}^1\times\mathbb{P}^1\rightarrow\mathbb{P}^1\times\mathbb{P}^1$. As a result, it is compatible with the Künneth decomposition via the cycle class map. The algebraic cell decomposition of $\mathbb{P}^1\times\mathbb{P}^1$ also yields that the real cycle class map is an isomorphism $\Ch^j(\mathbb{P}^1\times\mathbb{P}^1)\xrightarrow{\cong}\Hr^j(\mathbb{T})$ \cite[Proposition 5.3]{Hornbostel21}. Thus the image of the map 
\[
\overline{\gamma_\real}:\Ch^i(C)\otimes\Ch^j(\mathbb{P}^1\times\mathbb{P}^1)\rightarrow\Hr^i(C(\real))\otimes\Hr^j(\mathbb{T})
\] 
is equal to $\overline{\gamma_\real}(\Ch^i(C))\otimes\Hr^j(\mathbb{T})$. This proves that it surjects onto $\Hr^1(M)\otimes\Hr^1(\mathbb{T})$ and that its image into $\Hr^0(M)\otimes\Hr^2(\mathbb{T})$ is $\overline{\gamma_\real}(\Ch^0(C))\otimes\Hr^2(\mathbb{T})$. Since $C$ is (algebraically) connected, the equality $\overline{\gamma_\real}(\Ch^0(X))=\Z/2(s+t)$ holds which yields the lemma.
\end{proof}

Finally, to obtain an example with an affine variety, we use the following classical lemma.

\begin{lem}
Let $X$ be a projective real algebraic variety. Then there exists an affine open subscheme $U$ of $X$ such that $U(\real)=X(\real)$.
\end{lem}

\begin{proof}
We embed $X$ in $\mathbb{P}^N$ for some $N$. It then suffices to consider $U=X\setminus\{x_0^2+\cdots+x_N^2=0\}$.
\end{proof}

In the situation of the previous lemma, denoting by $j:U\hookrightarrow X$ the open immersion corresponding to $U$, the mod $2$ cycle class map induces a commutative square:
\[
\xymatrix{\Ch^n(X) \ar[r]^-{j^*} \ar[d]     & \Ch^n(U) \ar[d] \\
\Hr^n(X(\real)) \ar[r]_-{j(\real)^*} & \Hr^n(U(\real))}
\]
where $j^*$ is surjective and $j(\real)^*$ is an isomorphism. Thus we deduce that $\overline{\gamma_\real}(\Ch^n(X))=\overline{\gamma_\real}(\Ch^n(U))$. In particular, the previous lemma yields:

\begin{prop}
Let $U\subseteq C\times\mathbb{P}^1\times\mathbb{P}^1=X$ be an affine open subscheme such that $U(\real)=X(\real)$. Then $\mathrm{Sq}^1_L\Hr^2(U(\real))=\Hr^3(U(\real))$ but $\mathrm{Sq}^1_L\overline{\gamma_\real}(\Ch^2(U))\neq\Hr^3(U(\real))$.
\end{prop}

Consequently, the map $\Hr^3(U,\GW_3^2(\Lc))\rightarrow\mathrm{Coker}(\mathrm{Sq}_L)$ of Corollary \ref{cor:comparison_secondary_obstruction_group_first_part} is not injective in this case. Now since $U$ is affine of dimension $3$, we have an exact sequence 
\[
\Hr^1(U,\KMW_2(\Lc))\rightarrow\Hr^3(U,\GW_3^2(\Lc))\rightarrow\widetilde{\mathrm{GW}}^2(U,\Lc)\xrightarrow{e}{\CHW}^2(U,\Lc)\rightarrow 0
\] 
of abelian groups \eqref{eqn:GGW2ss}. Given a rank $2$ vector bundle $\Esc$ on $U$ of determinant $\Lc$, if its Euler class $e(\Esc)\in{\CHW}^2(U,\Lc)$ vanishes, then its class $[\Esc]-h\in \widetilde{\mathrm{GW}}^2(U,\Lc)$ defines an element in 
\[
\mathrm{Coker}(\Hr^1(U,\KMW_2(\Lc))\rightarrow\Hr^3(U,\GW_3^2(\Lc))).
\] 
Therefore in order to (dis)prove that a rank $2$ vector bundle with trivial primary obstruction in ${\CHW}^2(U,\Lc)$ actually splits, we would need to establish that the map 
\[
\Hr^1(U,\KMW_2(\Lc))\rightarrow\Hr^3(U,\GW_3^2(\Lc))
\] 
is (not) surjective. This is the object of the next section.

\subsection{Non-triviality of the first part of the secondary obstruction}

Now that we understand the group $\Hr^d(X,\mathbf{GW}_d^{d-1}(\omega_{X/\real}))$, we deal with the motivic secondary obstruction. 
The goal of the section is to prove the following theorem.
\begin{thm}\label{thm:main2}
There exists an affine open subscheme $U$ of $\mathbb{P}_\real^3$ such that $U(\real)=\mathbb{P}^3(\real)$ and a vector bundle $\mathscr{E}$ of rank $2$ on $U$ having the following properties:
\begin{itemize}[noitemsep,topsep=1pt]
\item $\det(\mathscr{E})$ is trivial.
\item $e(\mathscr{E})=0$.
\item $\mathscr{E}$ is (algebraically) indecomposable.
\item $\mathscr{E}(\real)\simeq \mathbb{P}^3(\real)\times \real^2$.
\end{itemize}
\end{thm}

Once again, the proof of this theorem will require some preparation. We start with an explicit invariant on $\mathrm{GW}^2(X,\omega_{X/k})$ with coefficients in $\Z/2$ for any smooth connected proper threefold.

\subsubsection{Push-forwards.}\label{subsection:pushforwards}
Let $k$ be a field of characteristic different from $2$ and let $X$ be a smooth proper threefold over $k$ with structure morphism $p$. Our aim is to make the push-forward map 
\[
p_*\colon \mathrm{GW}^2(X,\omega_{X/k})\to \mathrm{GW}^3(k)=\Z/2
\]
explicit (recall that by $4$-periodicity of Grothendieck--Witt groups, $\GWr^{2-3}(k)=\GWr^{-1}(k)$ is naturally isomorphic to $\GWr^{3}(k)$). Let then $\mathscr{P}$ be a locally free $\Osc_X$-module, endowed with a skew-symmetric form $\varphi\colon \mathscr{P}\to \mathscr{P}^\vee\otimes \omega_{X/k}$, where $\mathscr{P}^\vee=\mathcal{H}om_{\Osc_X}(\mathscr{P},\Osc)$. It will be convenient to see $\mathscr{P}$ as a complex concentrated in degree $3$ (we use the homological convention as is customary with (Grothendieck-)Witt groups). The push-forward of $(\mathscr{P},\varphi)$ is then explicitly given by the complex
\[
\xymatrix{\Hr^0(X,\mathscr{P})\ar[r]^-0 & \Hr^1(X,\mathscr{P})\ar[r]^-0 & \Hr^2(X,\mathscr{P})\ar[r]^-0 & \Hr^3(X,\mathscr{P})}
\]
with terms in degrees $3$ to $0$ (here and everywhere in Subsection \ref{subsection:pushforwards}, we use cohomology for the Zariski topology with coefficients in the coherent module $\mathscr{P}$). The form on this complex is given by the diagram
\[
\xymatrix{\Hr^0(X,\mathscr{P})\ar[r]^-0\ar[d]^-{\Hr^0(\varphi)} & \Hr^1(X,\mathscr{P})\ar[r]^-0\ar[d]^-{\Hr^1(\varphi)} & \Hr^2(X,\mathscr{P})\ar[r]^-0\ar[d]^-{\Hr^2(\varphi)} & \Hr^3(X,\mathscr{P})\ar[d]^-{\Hr^3(\varphi)} \\
\Hr^0(X,\mathscr{P}^\vee\otimes \omega_{X/k})\ar[r]^-0\ar[d]^-{\mathrm{can}} & \Hr^1(X,\mathscr{P}^\vee\otimes \omega_{X/k})\ar[r]^-0\ar[d]^-{\mathrm{can}}  & \Hr^2(X,\mathscr{P}^\vee\otimes \omega_{X/k})\ar[r]^-0\ar[d]^-{\mathrm{can}}  & \Hr^3(X,\mathscr{P}^\vee\otimes \omega_{X/k})\ar[d]^-{\mathrm{can}}  \\
\Hr^3(X,\mathscr{P})^*\ar[r]^-0 & \Hr^2(X,\mathscr{P})^*\ar[r]^-0 & \Hr^1(X,\mathscr{P})^*\ar[r]^-0 & \Hr^0(X,\mathscr{P})^*}
\]
where the isomorphism $\mathrm{can}\colon \Hr^i(X,\mathscr{P}^\vee\otimes \omega_{X/k})\to \Hr^{3-i}(X,\mathscr{P})^*$ is induced by Serre duality, i.e. the perfect pairing
\[
 \Hr^i(X,\mathscr{P}^\vee\otimes \omega_{X/k})\otimes \Hr^{3-i}(X,\mathscr{P})\xrightarrow{\cup}  \Hr^3(X, \omega_{X/k})\to k.
\]
The class of the above complex in $\mathrm{GW}^3(k)$ is in fact the hyperbolic form on the complex $\Hr^2(X,\mathscr{P})\xrightarrow{0}\Hr^3(X,\mathscr{P})$, whose class in $\mathrm{K}_0(k)$ is 
\[
[\Hr^3(X,\mathscr{P})]-[\Hr^2(X,\mathscr{P})]=[\Hr^0(X,\mathscr{P})^*]-[\Hr^2(X,\mathscr{P})]=[\Hr^0(X,\mathscr{P})]-[\Hr^2(X,\mathscr{P})]
\]
Since the hyperbolic map $H\colon \mathrm{K}_0(k)\to \mathrm{GW}^3(k)=\Z/2$ is the reduction modulo $2$, it follows that 
\[
p_*([\mathscr{P},\varphi])=\mathrm{dim}_k(\Hr^0(X,\mathscr{P}))+\mathrm{dim}_k(\Hr^2(X,\mathscr{P}))\pmod 2.
\]
For convenience, we write $h^i(X,\mathscr{P}):=\mathrm{dim}_k(\Hr^i(X,\mathscr{P}))\pmod 2$ in the sequel. We note that the homomorphism
\[
p_*\colon \mathrm{GW}^2(X,\omega_{X/k})\to \mathrm{GW}^3(k)
\]
is split surjective in case $X$ has a $k$-rational point. Indeed, if $x$ is such a point with closed immersion $\iota\colon x\to X$ the composite 
\[
\mathrm{GW}^3(k)\xrightarrow{\iota_*}\mathrm{GW}^2(X,\omega_{X/k})\xrightarrow{p_*} \mathrm{GW}^3(k)
\]
is the identity by functoriality of push-forwards. 

\subsubsection{The motivic Atiyah-Rees invariant.}

The computations of this section can be seen as a motivic analogue of the invariant constructed in \cite[Theorem 4.2]{Atiyah76}.
We consider the composite
\[
\chi\colon \mathrm{K}_0(\mathbb{P}^3_\real)\xrightarrow{H} \mathrm{GW}^2(\mathbb{P}^3_\real,\omega_{\mathbb{P}_\Rb^3/\Rb})\xrightarrow{p_*}\Z/2.
\]
where $\omega_{\mathbb{P}_\Rb^3/\Rb}=\Osc(-4)$. Explicitly, we have 
\[
\chi(\mathscr{P})=h^0(\mathbb{P}^3_\real,\mathscr{P})+h^2(\mathbb{P}^3_\real,\mathscr{P})+h^0(\mathbb{P}^3_\real,\mathscr{P}^\vee\otimes \Osc(-4))+h^2(\mathbb{P}^3_\real,\mathscr{P}^\vee\otimes \Osc(-4))
\]
The homomorphism is obviously surjective, e.g. since we can consider the class of a rational point $\iota:x\rightarrow\mathbb{P}_\Rb^3$ in $\mathrm{K}_0(\mathbb{P}^3)$ and the commutative diagram
\begin{equation}\label{eqn:rationalpoint}
\xymatrix{ \mathrm{K}_0(\Rb)\ar[r]^-{H}\ar[d]_-{\iota_*} &  \mathrm{GW}^3(\Rb)\ar[d]_-{\iota_*}\ar@{=}[r] & \Z/2\ar@{=}[d] \\
\mathrm{K}_0(\mathbb{P}^3_\real)\ar[r]_-{H} &  \mathrm{GW}^2(\mathbb{P}^3_\real,\Osc(-4))\ar[r]_-{p_*} & \Z/2}
\end{equation}
Let now $X\subset \mathbb{P}^3_\real$ be a smooth hypersurface as in Proposition \ref{prop:existencedegreed} and let $U:=\mathbb{P}^3_{\real}\smallsetminus X$. We suppose that $X$ is given by a homogeneous polynomial $g$ of even degree $d=2n$ such that $X(\real)=\emptyset$ and $\mathrm{CH}^1(X)=\Z$ generated by the class of $i^*\Osc(-1)$, where $i\colon X\to \mathbb{P}^3_{\real}$ is the closed inclusion. 

\begin{lem}\label{lem:chi_trivial_on_pushforward}
If $n$ is even, the composite 
\[
\mathrm{K}_0(X)\xrightarrow{i_*} \mathrm{K}_0(\mathbb{P}^3_{\real})\xrightarrow{\chi} \Z/2
\]
is trivial. 
\end{lem}

\begin{proof}
The filtration by codimension of support on $\Kr_0(X)$ is of the form \[\Kr^0(X)=\mathrm{F}^0\Kr_0(X)\supseteq\mathrm{F}^1\Kr_0(X)\supseteq\mathrm{F}^2\Kr_0(X)\supseteq\mathrm{F}^3\Kr_0(X)=0\] since $X$ has dimension $2$. We set $\Gr^p\Kr_0(X)=\mathrm{F}^p\Kr_0(X)/\mathrm{F}^{p+1}\Kr_0(X)$. For every $p$, there is a homomorphism $Z^p(X)\rightarrow\Gr^p\Kr_0(X)$ sending $[Y]$ to the class of the coherent $\Osc_X$-module $\Osc_Y$ for every codimension $p$ subvariety $Y$. This homomorphism factors through a surjective homomorphism $\CH^p(X)\rightarrow\Gr^p\Kr_0(X)$. Therefore $\Gr^2\Kr_0(X)=\mathrm{F}^2\Kr_0(X)$ is generated by the classes of closed points. Moreover, $\CH^1(X)$ is generated by the class of $i^*\Osc(-1)$ so $\Gr^1\Kr_0(X)$ is generated by the image of this class. Therefore the exact sequence \[0\rightarrow\mathrm{F}^2\Kr_0(X)\rightarrow\mathrm{F}^1\Kr_0(X)\rightarrow\Gr^1\Kr_0(X)\rightarrow 0\] shows that $\mathrm{F}^1\Kr_0(X)$ is generated by the classes of closed points together with $[i^*\Osc(-1)]-[\Osc_X]$. Finally $\Gr^0\Kr_0(X)=\mathbb{Z}$ is generated by $[\Osc_X]$ so the exact sequence \[0\rightarrow\mathrm{F}^1\Kr_0(X)\rightarrow\mathrm{F}^0\Kr_0(X)\rightarrow\Gr^0\Kr_0(X)\rightarrow 0\] shows that $\mathrm{K}_0(X)$ is generated by the classes of closed points, together with $[i^*\Osc(-1)]-[\Osc_X]$ and $[\Osc_X]$.

Using the (Koszul) exact sequence of sheaves on $\mathbb{P}^3_{\real}$
\[
0\rightarrow \Osc(-2n)\xrightarrow{g}\Osc\to i_*\Osc_X\rightarrow 0,
\]
we see that $i_*[\Osc_X]=[\Osc]-[\Osc(-2n)]$ and that $i_*i^*[\Osc(-1)]=[\Osc(-1)]-[\Osc(-2n-1)]$ in $\Kr_0(\mathbb{P}_\Rb^3)$. Recall the classical computation of the cohomology of projective space $\mathbb{P}_\Rb^N$: we have $\dim_\Rb\Hr^i(\mathbb{P}_\Rb^N,\Osc(m))=0$ if $i\notin\{0,N\}$, $\dim_\Rb\Hr^0(\mathbb{P}_\Rb^N,\Osc(m))=0$ if $m<0$ and $\dim_\Rb\Hr^0(\mathbb{P}_\Rb^N,\Osc(m))=\binom{N+m}{m}$ if $m\geq 0$ (\cite[\href{https://stacks.math.columbia.edu/tag/01XT}{Tag 01XT}]{Stacks22}). We then have 
\[
\chi([\Osc]) = h^0(\mathbb{P}^3_{\real},\Osc)+h^2(\mathbb{P}^3_{\real},\Osc)+h^0(\mathbb{P}^3_{\real},\Osc(-4))+h^2(\mathbb{P}^3_{\real},\Osc(-4))=\binom{3+0}{0}=1,
\]
while
\begin{eqnarray*}
\chi([\Osc(-2n)]) & = & h^0(\mathbb{P}^3_{\real},\Osc(-2n))+h^2(\mathbb{P}^3_{\real},\Osc(-2n))+h^0(\mathbb{P}^3_{\real},\Osc(2n-4))+h^2(\mathbb{P}^3_{\real},\Osc(2n-4)) \\
& = & \binom {2n-4+3}{2n-4}\pmod 2 \\
& = & \binom {2n-1}{3}\pmod 2 \\
& = & \frac{(2n-1)(2n-2)(2n-3)}{2\cdot 3} \pmod 2 \\
& = & \frac{(2n-1)(n-1)(2n-3)}{3}\pmod 2 \\
& = & 1\pmod 2
\end{eqnarray*}
since $3=1\pmod 2$ and $n=0\pmod 2$ by assumption. It follows that $\chi(i_*[\Osc_X])=0$. 
The same kind of computations yield $\chi([\Osc(-1)])=0$, and 
\[
\chi([\Osc(-2n-1)])=\binom {2n}3\pmod 2=0\pmod 2
\]
showing that $\chi(i_*i^*[\Osc(-1)])=0$. Therefore $\chi\circ i_*:\Kr_0(X)\rightarrow\mathbb{Z}/2$ vanishes on $[\Osc_X]$ and on $[i^*\Osc(-1)]-[\Osc_X]$.

It now remains to be seen that the class in $\Kr_0(X)$ of any closed point of $X$ is mapped to $0$ under $\chi\circ i_*$. Now, any closed point of $X$ is complex, and we consider $x:\Spec(\cplx)\to X$  such a point. We have a commutative diagram
\[
\xymatrix{\mathrm{K}_0(\cplx)\ar[r]\ar[d]_-{x_*} & \mathrm{GW}^3(\cplx)\ar@{=}[r]\ar[d]_-{x_*} &  \mathrm{GW}^3(\cplx)\ar@{=}[d] \\
\mathrm{K}_0(X)\ar[r]\ar[d]_-{i_*} &  \mathrm{GW}^1(X,\omega_{X/k})\ar[r]\ar[d]_-{i_*} &  \mathrm{GW}^3(\cplx)\ar[d] \\
\mathrm{K}_0(\mathbb{P}^3_{\real})\ar[r] &  \mathrm{GW}^2(\mathbb{P}^3_{\real},\Osc(-4))\ar[r]_-{p_*} &  \mathrm{GW}^3(\real)}
\]
and we conclude by observing that the push-forward map $\mathrm{GW}^3(\cplx)\to \mathrm{GW}^3(\real)$ is trivial by Example \ref{exe:pushfoward_GW_even_degree}.
\end{proof}

\begin{cor}\label{cor:invariant_chi_at_rational_point}
If $n$ is even, the surjective homomorphism $\chi\colon \mathrm{K}_0(\mathbb{P}^3_{\real})\to \Z/2$ induces a surjective map $\overline{\chi}\colon \mathrm{K}_0(U)\to \Z/2$. If $x\in U(\real)$ is any rational point, we have $\overline{\chi}([\Osc_x])=1$.
\end{cor}

\begin{proof}
Indeed we have an exact sequence \[\mathrm{K}_0(X)\xrightarrow{i_*}\mathrm{K}_0(\mathbb{P}^3_{\real})\to \mathrm{K}_0(U)\to 0\] of abelian groups: since $\chi\circ i_*=0$ by Lemma \ref{lem:chi_trivial_on_pushforward}, $\chi$ indeed descends to $\Kr_0(U)$. To show that $\overline{\chi}([\Osc_x])=1$ for $x\in U(\Rb)$, it suffices by definition to show that $[\Osc_x]\in\Kr_0(\mathbb{P}_\Rb^3)$ is mapped to $1$ under $\chi$. This follows from \eqref{eqn:rationalpoint}. 
\end{proof}

\begin{lem}\label{lem:trivialchi}
Let $F\colon \mathrm{GW}^1(U,\Osc(-4))\to \mathrm{K}_0(U)$ be the forgetful map. Then, the composite
\[
\mathrm{GW}^1(U,\Osc(-4))\xrightarrow{F} \mathrm{K}_0(U)\xrightarrow{\overline{\chi}} \Z/2
\] 
is trivial. 
\end{lem}

\begin{proof}
The group $\mathrm{GW}^1(U,\Osc(-4))$ is generated by classes of the form $[\mathscr{P}_\bullet,\varphi]$, where $\mathscr{P}_{\bullet}$ a bounded complex of locally free $\Osc_U$-modules, and $\varphi\colon \mathscr{P}_{\bullet}\to \mathscr{P}_\bullet^\vee\otimes \Osc(-4)[1]$ is a symmetric quasi-isomorphism. Using \cite[Theorem 7.1]{Walter03}, we may even suppose that $\mathscr{P}_\bullet$ is of the form
\[
0\to \mathscr{P}_1\xrightarrow{d}\mathscr{P}_0\to 0
\] 
in which case the quasi-isomorphism takes the form
\[
\xymatrix{0\ar[r] &  \mathscr{P}_1\ar[r]^-{d}\ar[d]_-{\varphi_1} & \mathscr{P}_0\ar[r]\ar[d]^-{\varphi_0} &  0 \\
0\ar[r] &  \mathscr{P}_0^\vee\otimes \Osc(-4)\ar[r]^-{d} & \mathscr{P}_1^\vee\otimes\Osc(-4)\ar[r] &  0}
\]
In $\mathrm{K}_0(U)$, this quasi-isomorphism yields an equality 
\[
[\mathscr{P}_0\oplus (\mathscr{P}_0^\vee\otimes \Osc(-4)) ]=[\mathscr{P}_1\oplus (\mathscr{P}_1^\vee\otimes \Osc(-4)) ]
\]
As $F([\mathscr{P}_\bullet,\varphi])=[\mathscr{P}_0]-[\mathscr{P}_1]$, we conclude easily.
\end{proof}

\begin{cor}
Let $x\in U(\real)$ be a rational point. Then, the class $[\Osc_x]\in \mathrm{K}_0(U)$ yields a non trivial class $[x]$ in $\mathrm{GW}^2(U,\Osc(-4))$ under the hyperbolic functor
\[
H\colon \mathrm{K}_0(U)\to \mathrm{GW}^2(U,\Osc(-4)). 
\]
Further, the image of $[x]$ under the first Borel class (which agrees with the Euler class) 
\[
\mathrm{GW}^2(U,\Osc(-4))\xrightarrow{b_1}\CHW^2(U,\Osc(-4))
\]
is trivial.  
\end{cor}

\begin{proof}
Since the map $\overline{\chi}:\Kr_0(U)\rightarrow\Zb/2$ takes the value $1$ at the class $[\Osc_x]$ of the rational point $x$ in $\Kr_0(X)$ by Corollary \ref{cor:invariant_chi_at_rational_point}, and since $\overline{\chi}\circ F=0$ by Lemma \ref{lem:trivialchi}, the class $[\Osc_x]$ does \emph{not} lie in the image of the forgetful homomorphism $F:\GWr^1(U,\Osc(-4))\rightarrow\Kr_0(U)$. The Karoubi periodicity exact sequence 
\[
\mathrm{GW}^1(U,\Osc(-4))\xrightarrow{F}\mathrm{K}_0(U)\xrightarrow{H} \mathrm{GW}^2(U,\Osc(-4))
\]
then shows that $H([\Osc_x])\neq 0$. For the second statement, we consider the filtration of $\mathrm{F}^{\bullet}\mathrm{GW}^2(U,\Osc(-4))$ by codimension of support. The Euler class induces an isomorphism
\[
\mathrm{F}^{2}\mathrm{GW}^2(U,\Osc(-4))/\mathrm{F}^{3}\mathrm{GW}^2(U,\Osc(-4))\xrightarrow{b_1}\CHW^2(U,\Osc(-4))
\]
from which the result follows.
\end{proof}

\begin{rem}
In fact, the exact sequence \eqref{eqn:GGW2ss} 
\[
\Z/2=\Hr^3(X,\mathbf{GW}_3^2(\Osc(-4)))\to \widetilde{\mathrm{GW}}^2(U,\Osc(-4))\xrightarrow{b_1}\CHW^2(U,\Osc(-4))\to 0 
\]
and the above corollary show that the left-hand morphism is injective. In other words, the differential 
\[
d_2^{1,0}\colon \Hr^1(U,\KMW_2(\Osc(-4)))\to \Hr^3(X,\mathbf{GW}_3^2(\Osc(-4)))
\]
in \eqref{eqn:GGW2ss} is trivial. On the other hand, if $n$ is odd, then one can show that the homomorphism $\mathrm{K}_0(X)\xrightarrow{i_*} \mathrm{K}_0(\mathbb{P}^3_{\real})\xrightarrow{\chi} \Z/2$ of Lemma \ref{lem:trivialchi} is in fact onto. Consequently, 
\[
d_2^{1,0}\colon\Hr^1(U,\KMW_2(\Osc(-4)))\to \Hr^3(X,\mathbf{GW}_3^2(\Osc(-4)))
\]
is non trivial in that case. So, the cohomological operation $\mathrm{K}(\KMW_2,1)\to \mathrm{K}(\mathbf{GW}_3^2,3)$ given by the spectral sequence is in general non trivial, even for threefolds.
\end{rem}

We can finally prove the main result of this section. 

\begin{proof}[Proof of Theorem \ref{thm:main2}]
The Grothendieck-Witt groups $\mathrm{GW}^i(U,\Osc(-4))$ are canonically isomorphic to the Grothendieck-Witt groups $\mathrm{GW}^i(U)$, and similarly for Chow-Witt groups. From the above corollary, we obtain a non trivial element of $\mathrm{GW}^2(U)$ which has a trivial Euler class. Since $U$ is affine of dimension $3$, every element of $\mathrm{GW}^2(U)$ is the class of an actual rank $2$ symplectic bundle. This holds in particular for the class $[x]$ associated to a rational point, and we obtain a non trivial (and thus indecomposable) symplectic bundle $\mathscr{E}$ with $b_1( \mathscr{E})=e( \mathscr{E})=0$. Thus $\Esc(\Rb)$ has trivial Euler class and trivial determinant.  Therefore it is trivial over the threefold $\mathbb{P}^3(\Rb)$ by Lemma \ref{lem:top_splitting_over_threefolds} as required.
\end{proof}

\begin{rem}
One can be a bit more explicit about the vector bundle $\mathscr{E}$. The class of the rational point $x$ in $\mathrm{K}_0(U)$ is of the form
\[
[\Osc_x]=[\Osc]-3[\Osc(-1)]+3[\Osc(-2)]-[\Osc(-3)]
\]
and consequently 
\[
H([\Osc_x])=[\Osc\oplus \Osc(-4)]-3[\Osc(-1)\oplus \Osc(-3)]+3[\Osc(-2)\oplus \Osc(-2)]-[\Osc(-3)\oplus \Osc(-1)]
\]
in $\widetilde{\mathrm{GW}}^2(U,\Osc(-4))$. Tensoring with $\Osc(2)$, we obtain the following class:
\[
[\mathscr{E}]=[\Osc(2)\oplus \Osc(-2)]-4[\Osc(1)\oplus \Osc(-1)]+3[\Osc\oplus \Osc].
\]
in $\widetilde{\mathrm{GW}}^2(U)$.
\end{rem}

\subsection{The second part of the secondary obstruction}
Recall once again that the homotopy sheaf relevant for computing the motivic secondary obstruction is of the form \eqref{eqn:firstnontrivialhomotopysheaf}
\[
\KM_{n+2}/24 \longrightarrow \piaone_n(\A^n\smallsetminus\{0\}) \longrightarrow \mathbf{GW}_{n+1}^n
\]
for $n\geq 4$ and slightly more complicated for $n=2,3$ \cite{Asok12a,Asok12c}. In the previous sections, we've studied the first part of the secondary obstruction, and we now study the other part, i.e., the one coming from $\KM_{n+2}/24$. We therefore compute the group $\Hr^{d}(X,\KM_{d+1}/a)$ where $a\in\Z$ and $d$ is the dimension of $X$.

\begin{lem}\label{lem:odddivisible}
Let $X$ be a smooth $\Rb$-scheme of dimension $d$. The group $\Hr^d(X_\Cb,\KM_{d+1})$ is then divisible. If $n$ is an odd integer, then $\Hr^d(X,\KM_{d+1})$ is $n$-divisible and thus $\Hr^d(X,\KM_{d+1}/n)=0$.
\end{lem}

\begin{proof}
Since $X$ has dimension $d$, the Rost--Schmid complex of $\KM_{d+1}$ exhibits $\Hr^d(X_\Cb,\KM_{d+1})$ as a quotient of \[\Cr_\RS^d(X_\Cb,\KM_{d+1})=\bigoplus_{x\in X_\Cb^{(d)}}\mathrm{K}_1(\kappa(x))\] where $\mathrm{K}_1(F)=F^\times$ for any field $F$. The points in $X_\Cb^{(d)}$ are closed so $\kappa(x)=\Cb$ for every $x\in X_\Cb^{(d)}$, thus the group $\mathrm{K}_1(\kappa(x))=\Cb^\times$ is divisible as $\Cb$ is algebraically closed. We conclude that $\Cr_\RS^d(X_\Cb,\KM_{d+1})$, as a direct sum of divisible groups, is divisible hence so is its quotient $\Hr^d(X_\Cb,\KM_{d+1})$. Now if $x\in X^{(d)}$, then again $x$ is closed so the group $\mathrm{K}_1(\kappa(x))$ is isomorphic to $\Cb^\times$ or to $\Rb^\times$, both of which are $n$-divisible for $n$ odd and we see again that $\Hr^d(X,\KM_{d+1})$ is $n$-divisible. Finally since $X$ has dimension $d$, the short exact sequence \[\KM_{d+1}\xrightarrow{\cdot n}\KM_{d+1}\to\KM_{d+1}/n\to 0\] of sheaves induces an exact sequence \[\Hr^d(X,\KM_{d+1})\xrightarrow{\cdot n}\Hr^d(X,\KM_{d+1})\to\Hr^d(X,\KM_{d+1}/n)\to 0\] of abelian groups. Since $\Hr^d(X,\KM_{d+1})$ is $n$-divisible, we conclude that $\Hr^d(X,\KM_{d+1}/n)=0$ as required.
\end{proof}

To understand $\Hr^d(X,\KM_{d+1}/a)$ when $a$ is even, we consider the projection map $f\colon X_{\cplx} \to X$. It is a degree $2$ étale cover.

\begin{cor}
Let $X$ be a smooth affine real variety of dimension $d$, and let $a\neq 0$ be an even integer. Then, the epimorphism $\KM_{d+1}/a\to \KM_{d+1}/2$ induces an isomorphism
\[
\Hr^d(X,\KM_{d+1}/a)\to \Hr^d(X,\KM_{d+1}/2)
\] of abelian groups.
\end{cor}

\begin{proof}
If $a=2^m\cdot n$ with $n$ odd, we have an exact sequence
\[
\KM_{d+1}/n\xrightarrow{\cdot 2^m}\KM_{d+1}/a\to \KM_{d+1}/2^m\to 0.
\]
Since $X$ has dimension $d$, this exact sequence induces an exact sequence \[\Hr^d(X,\KM_{d+1}/n)\xrightarrow{\cdot 2^m}\Hr^d(X,\KM_{d+1}/a)\rightarrow\Hr^d(X,\KM_{d+1}/2^m)\rightarrow 0\] by the argument of the proof of Lemma \ref{lem:odddivisible}. The group $\Hr^d(X,\KM_{d+1}/n)$ vanishes by Lemma \ref{lem:odddivisible}. Therefore the map $\Hr^d(X,\KM_{d+1}/a)\rightarrow\Hr^d(X,\KM_{d+1}/2^m)$ is an isomorphism so that we may replace $a$ by $2^m$ and suppose that $a=2^m$. Now we use the pull-backs and push-forward along $f:X_\Cb\rightarrow X$ to obtain a composite
\[
\Hr^d(X,\KM_{d+1}/2^{m})\xrightarrow{f^*}\Hr^d(X_{\cplx},\KM_{d+1}/2^{m})\xrightarrow{f_*}\Hr^d(X,\KM_{d+1}/2^{m})
\] 
which is multiplication by $2$ since $f$ is a degree $2$ étale cover. Since $\Hr^d(X_\Cb,\KM_{d+1})$ is divisible by Lemma \ref{lem:odddivisible}, the group $\Hr^d(X_\Cb,\KM_{d+1}/2^m)$ vanishes by the same argument as in the proof of this lemma. We conclude that $2\Hr^d(X,\KM_{d+1}/2^{m})=0$. Now again, as observed in the proof of Lemma \ref{lem:odddivisible}, multiplication by $2$ on the sheaf $\KM_{d+1}/2^m$ induces an exact sequence \[\Hr^d(X,\KM_{d+1}/2^m)\xrightarrow{\cdot 2}\Hr^d(X,\KM_{d+1}/2^m)\rightarrow\Hr^d(X,\KM_{d+1}/2)\rightarrow 0\] so that $\Hr^d(X,\KM_{d+1}/2^m)\rightarrow\Hr^d(X,\KM_{d+1}/2)$ is an isomorphism. This completes the proof.
\end{proof}

\begin{cor}\label{cor:eventotwo}
Let $X$ be a smooth affine real variety of dimension $d$, and let $a\neq 0$ be an even integer. Then the real realization map 
\[
\overline{\gamma_\real}:\Hr^d(X,\KM_{d+1})\rightarrow\Hr^d(X(\real),\Z/2)
\]
induces an isomorphism $\Hr^d(X,\KM_{d+1}/a)\rightarrow\Hr^d(X(\real),\Z/2)$.
\end{cor}

\section{Positive results}
In this section, we gather positive results about the main question, using the results of the above section.

\subsection{Varieties without real points}

\begin{thm}\label{thm:norealpoints}
Let $X$ be a smooth affine real variety of dimension $d\geq 3$ such that $X(\real)=\emptyset$. If $\Esc$ is a locally free $\Osc_X$-module of rank $d-1$, then $\Esc\simeq \Esc'\oplus \Osc_X$ if and only if $c_{d-1}(\Esc)=0$.
\end{thm}

\begin{proof}
We just have to prove that the primary and secondary obstructions vanish. For the first one, we note that the conditions of Proposition \ref{prop:corank_one_primary_obstruction} are trivially satisfied in that case. Thus, our assumption that $c_{d-1}(\Esc)=0$ implies the vanishing of the (algebraic) Euler class. We are thus left with proving that the secondary obstruction also vanishes. If $d\geq 5$, we have an exact sequence 
\[
\KM_{d+1}/24\to \piaone_{d-1}(\A^{d-1}\smallsetminus\{0\})\to \mathbf{GW}_{d}^{d-1}
\]
and we know that $\Hr^d(X,\KM_{d+1}/24)=\Hr^d(X(\real),\Z/2)=0$. On the other hand, Corollary \ref{cor:comparison_secondary_obstruction_group_first_part} shows that $\Hr^d(X,\mathbf{GW}_{d}^{d-1}(\Lc))=0$ as well, for any line bundle $\Lc$ over $X$. The secondary obstruction then vanishes automatically and the result is proved. If $d=4$, the relevant exact sequence of sheaves is of the form
\[
\KM_{5}/24\times_{\overline{\mathbf{I}^5}}\mathbf{I}^5 \to \piaone_{3}(\A^{3}\smallsetminus\{0\})\to \mathbf{GW}_{4}^{3}
\]
Using the exact sequence of sheaves $0\to \mathbf{I}^6\to \mathbf{I}^5\to \overline{\mathbf{I}}^5\to 0$, we see that the left-hand sheaf fits into an exact sequence
\[
0\to \mathbf{I}^6\to \KM_{5}/12\times_{\overline{\mathbf{I}^5}}\mathbf{I}^5\to \KM_{5}/12\to 0
\]
and it suffices to observe that $\Hr^d(X,\mathbf{I}^6(\Lc))=0$ for any line bundle $\Lc$ over $X$ by Theorem \ref{theo:jacobson} to conclude as in the previous case. The same argument works if $d=3$, using this time the exact sequence
\[
\KM_{4}/12\times_{\overline{\mathbf{I}^4}}\mathbf{I}^4 \to \piaone_{2}(\A^{2}\smallsetminus\{0\})\to \mathbf{GW}_{3}^{2}.
\]
\end{proof}

\begin{rem}\label{rem:cohomologicaldimensiondminus3}
Similar computations of the group $\mathrm{H}^d(X,\KMW_{d+1})$ appear in \cite{Das17} (see below), but they assume that $X(\real)$ is orientable and nonempty. Consequently we cannot use their result in our situation.
\end{rem}

\begin{rem}
In fact, Theorem \ref{thm:norealpoints} is valid under the less restrictive hypothesis that $\Hr^d(X(\Rb),\Zb/2)$ and $\Hr^{d-1}(X(\Rb),\Zb(\Lc(\Rb)))$ vanish for every line bundle $\Lc$ on $X$. Indeed, under this hypothesis, the group $\Hr^{d-1}(X,\Ibf^d(\Lc))$ vanishes by Proposition \ref{prop:colliotsch} so the map $\widetilde{\CH}^{d-1}(X,\Lc)\to\CH^d(X)$ is injective and thus the Euler class vanishes under the assumption of Theorem \ref{thm:norealpoints} that the top Chern class is zero. It then suffices to prove that the secondary obstruction is automatically trivial, and thus that the groups $\Hr^d(X,\KM_{d+1}/a)$ and $\Hr^d(X,\mathbf{GW}_d^{d-1}(\Lc))$ vanish for every line bundle $\Lc$ on $X$. This is true by Corollary \ref{cor:comparison_secondary_obstruction_group_first_part} and Corollary \ref{cor:eventotwo} under the assumption that $\Hr^d(X(\Rb),\Zb/2)=0$.
\end{rem}

\subsection{Varieties with open real locus}

\begin{thm}\label{thm:nocompactcomponent}
Let $X$ be a smooth affine real variety of dimension $d\geq 3$ such that $X(\real)$ has no connected compact component. If $\Esc$ is a locally free $\Osc_X$-module of rank $d-1$, then $\Esc\simeq \Esc'\oplus \Osc_X$ if and only if $e(\Esc)=0$.
\end{thm}

\begin{proof}
As in the proof of the previous theorem, it suffices to prove that the secondary obstruction vanishes. The first part vanishes since $\Ch^d(X)=0$ by \cite[Theorem 1.3.(b)]{Colliot96}, and the second part vanishes since $\Hr^d(X(\real),\Z)=\Hr^d(X(\real),\Z/2)=0$ as $X(\Rb)$ has no compact connected component.
\end{proof}

\begin{rem}
In case $X(\Rb)$ is orientable, we could instead use \cite[Theorem 1.1]{Das17} to conclude. 
\end{rem}


{\begin{footnotesize}
\raggedright
\bibliographystyle{alpha}
\bibliography{splitting}
\end{footnotesize}}

\Addresses

\end{document}